\documentclass[a4paper,reqno,11pt]{amsart}
\usepackage{amssymb}
\usepackage{amssymb}
\usepackage{amsfonts}
\usepackage{amsmath}
\usepackage{graphicx}
\usepackage[numbers,sort&compress]{natbib}
\usepackage[hidelinks]{hyperref}
\usepackage[left=3.2cm,right=3.2cm,bottom=3cm,top=3cm]{geometry}
\usepackage{color}

\newtheorem{theorem}{Theorem}[section]

\newtheorem{fact}[theorem]{Fact}

\newtheorem{corollary}[theorem]{Corollary}

\newtheorem{definition}[theorem]{Definition}

\newtheorem{lemma}[theorem]{Lemma}

\newtheorem{proposition}[theorem]{Proposition}
\newtheorem*{proposition*}{Proposition}

\DeclareMathOperator{\Aut}{Aut}
\DeclareMathOperator{\Mult}{Mult}

\begin{document}
\title[Classification of operator algebraic varieties]{The classification
problem for operator algebraic varieties and their multiplier algebras}
\author{Michael Hartz}
\address{Department of Pure Mathematics, University of Waterloo, Waterloo,
ON N2L 3G1, Canada}
\curraddr{Department of Mathematics, Washington University in St Louis, One
Brookings Drive, St. Louis, MO 63130}
\email{mphartz@wustl.edu}
\author{Martino Lupini}
\address{Fakult\"{a}t f\"{u}r Mathematik, Universit\"{a}t Wien,
Oskar-Morgenstern-Platz 1, Room 02.126, 1090 Wien, Austria}
\curraddr{Mathematics Department\\
California Institute of Technology\\
1200 E. California Blvd\\
MC 253-37\\
Pasadena, CA 91125}
\email{lupini@caltech.edu}
\thanks{M.H. was partially supported by an Ontario Trillium Scholarship.
M.L. was supported by the York University Susan Mann Dissertation
Scholarship and by the ERC Starting grant no.\ 259527 of Goulnara
Arzhantseva. This work was initiated during a visit of the first named
author to the Fields Institute in March 2015. The hospitality of the
Institute is gratefully acknowledged.}
\dedicatory{ }

\begin{abstract}
We study from the perspective of Borel complexity theory the classification
problem for multiplier algebras associated with operator algebraic
varieties. These algebras are precisely the multiplier algebras of
irreducible complete Nevanlinna-Pick spaces. We prove that these algebras
are not classifiable up to algebraic isomorphism using countable structures
as invariants. In order to prove such a result, we develop the theory of
turbulence for Polish groupoids, which generalizes Hjorth's turbulence
theory for Polish group actions. We also prove that the classification
problem for multiplier algebras associated with varieties in a finite
dimensional ball up to isometric isomorphism has maximum complexity among
the essentially countable classification problems. In particular, this shows
that Blaschke sequences are not smoothly classifiable up to conformal
equivalence via automorphisms of the disc.
\end{abstract}

\maketitle

\section{Introduction\label{Section:introduction}}

Let $d\in \mathbb{N}$ be a natural number and let $\mathbb{B}_{d}$ denote
the open unit ball in $\mathbb{C}^{d}$. The \emph{Drury-Arveson space} $%
H_{d}^{2}$ is the completion of the space of complex polynomials $\mathbb{C}%
\left[ z_{1},\ldots ,z_{d}\right] $ in the variables $z_{1},\ldots ,z_{d}$
with respect to the inner product defined on monomials by%
\begin{equation*}
\left\langle z_{1}^{\alpha _{1}}\cdots z_{d}^{\alpha _{d}},z_{1}^{\beta
_{1}}\cdots z_{d}^{\beta _{d}}\right\rangle =\left\{ 
\begin{array}{cc}
\frac{\alpha _{1}!\cdots \alpha _{d}!}{\left( \alpha _{1}+\cdots +\alpha
_{d}\right) !} & \text{if }\alpha _{i}=\beta _{i}\text{ for }1\leq i\leq d%
\text{;} \\ 
0 & \text{otherwise.}%
\end{array}%
\right.
\end{equation*}%
This space has been the focus of intensive study for over a decade. If $d=1$%
, then the Drury-Arveson space is the classical Hardy space on the unit
disc, which plays an important role in the theory of operators on Hilbert
space. From the point of view of operator theory, $H^2_d$ appears to be the
correct generalization of the classical Hardy space to several variables.
For a comprehensive treatment of the Drury-Arveson space, the reader is
referred to \cite{arveson_subalgebras_1998} and to the survey article \cite%
{shalit_operator_2013}.

One can identify $H_{d}^{2}$ with a Hilbert space of analytic functions on $%
\mathbb{B}_{d}$ (see, for example, \cite[Section 3]{shalit_operator_2013}).
An \emph{operator algebraic variety} is the set of common zeros in $\mathbb{B%
}_{d}$ of some subset of $H_{d}^{2}$. Observe that any algebraic variety is
an operator algebraic variety. In turn, an operator algebraic variety is in
particular an analytic variety.

A similar construction can be performed when $d=\infty $. In this case $%
\mathbb{C}^{\infty }$ should be regarded as the infinite-dimensional
separable Hilbert space, and $\mathbb{B}_{\infty }$ as the norm open unit
ball of $\mathbb{C}^{\infty }$. The Drury-Arveson space $H_{\infty }^{2}$ is
defined as above as the completion of the algebra of complex polynomials $%
\mathbb{C}\left[ z\right] $ in the sequence of variables $z_{n}$ for $n\in 
\mathbb{N}$. In this context a complex-valued function defined on an open
subset of $\mathbb{B}_{\infty }$ is \emph{analytic }if it admits a uniformly
convergent power series representation.

The \emph{multiplier algebra} of $H^2_d$ is defined by 
\begin{equation*}
\Mult(H_{d}^{2})=\{\varphi :\mathbb{B}_{d}\rightarrow \mathbb{C}:\varphi
\cdot f\in H_{d}^{2}\text{ for all }f\in H_{d}^{2}\}.
\end{equation*}
Every element $\varphi \in \Mult(H^2_d)$ induces a bounded multiplication
operator $M_\varphi$ on $H^2_d$. The identification of $\varphi$ with $%
M_\varphi$ allows one to regard $\Mult(H^2_d)$ as an algebra of bounded
linear operators on the Hilbert space $H^2_d$.

For an operator algebraic variety $V\subset \mathbb{B}_{d}$, we consider the
algebra 
\begin{equation*}
\mathcal{M}_{V}=\{f |_{V}:f\in \Mult(H_{d}^{2})\}.
\end{equation*}%
Then $\mathcal{M}_{V}$ is an algebra of functions on $V$, and, in fact, can
be regarded as the multiplier algebra of a Hilbert space of functions on $V$%
. Observe also that $\mathcal{M}_V$ can be identified with a quotient of $%
\Mult(H^2_d)$. These algebras are of particular interest because of a
theorem of Agler and McCarthy \cite{agler_complete_2000}, according to which
every multiplier algebra of an irreducible complete Nevanlinna-Pick space
can be identified with such an algebra.

The problem of classifying the multiplier algebras $\mathcal{M}_{V}$ has
attracted considerable attention in the last few years \cite%
{davidson_isomorphism_2011, hartz_topological_2012,
davidson_multipliers_2014, davidson_operator_2015,
hartz_isomorphism_2015,alpay_hilbert_2003,kerr_isomorphism_2013,mccarthy_spaces_2015,arcozzi_carleson_2008}%
. In particular, we refer the reader to the survey article \cite%
{salomon_isomorphism_2014}. It is proved in \cite{davidson_operator_2015}%
---see also \cite[Section 4]{salomon_isomorphism_2014}---that for $d\in 
\mathbb{N}$ and operator algebraic varieties $V,W$ in $\mathbb{B}_{d}$, the
corresponding multiplier algebras $\mathcal{M}_{V}$ and $\mathcal{M}_{W}$
are (completely) isometrically isomorphic if and only if $V$ and $W$ are $%
\Aut(\mathbb{B}_{d})$\emph{-conformally equivalent}. This means that there
exists an automorphism $F\in \Aut(\mathbb{B}_{d})$ mapping $V$ onto $W$.
Here, $\Aut(\mathbb{B}_{d})$ denotes the group of biholomorphic maps of $%
\mathbb{B}_{d}$ onto itself. Similar conclusions hold for $d=\infty $ as
long as $V$ and $W$ have the same \emph{affine codimension }\cite[Subsection
4.1]{salomon_isomorphism_2014}.

The situation for algebraic isomorphism is far less clear. (It should be
noted that $\mathcal{M}_{V}$ and $\mathcal{M}_{W}$, being commutative and
semisimple, are algebraically isomorphic if and only if they are isomorphic
as Banach algebras.) Under some mild assumption on the varieties and for $%
d<\infty $, it was shown in \cite{davidson_operator_2015} that if $\mathcal{M%
}_{V}$ and $\mathcal{M}_{W}$ are algebraically isomorphic, then $V$ and $W$
are biholomorphically equivalent. The converse fails in multiple ways; see 
\cite{davidson_multipliers_2014}. On the other hand, the results of \cite%
{davidson_isomorphism_2011,hartz_topological_2012} show that if $V,W$ are
homogeneous algebraic varieties in $\mathbb{B}_{d}$ for $d$ finite, then $%
\mathcal{M}_{V}$ and $\mathcal{M}_{W}$ are algebraically isomorphic if and
only if there exists an invertible linear map of $\mathbb{C}^{d}$ mapping $V$
onto $W$. Recently, this result was extended in \cite{hartz_isomorphism_2015}
by studying the algebras $\mathcal{M}_{V}$ from a different point of view.

In this paper, we investigate the classification problem for multiplier
algebras $\mathcal{M}_{V}$ from the perspective of Borel complexity theory.
Our main result is that the classification problem for multiplier algebras $%
\mathcal{M}_{V}$ up to algebraic isomorphism is intractable in the sense of
Borel complexity theory.

\begin{theorem}
\label{Theorem:main}The multiplier algebras $\mathcal{M}_V$, where $V$ is an
operator algebraic variety in $\mathbb{B}_{\infty }$, are not classifiable
by countable structures up to algebraic isomorphism.
\end{theorem}

This means that there is no explicit way to classify the multiplier algebras 
$\mathcal{M}_{V}$ using countable structures as complete invariants. In
particular this rules out classification by K-theoretic data, modulo the
routine check that the assignment of its K-theory to a Banach algebra is
given by a Borel map; see \cite[Chapters 5, 8, 9]{blackadar_k-theory_1998}.
A more precise version of the statement of Theorem \ref{Theorem:main} will
be given in Section \ref{Section:complexity}. The proof of Theorem \ref%
{Theorem:main} is presented in Sections \ref{Section:reproducing} and \ref%
{Section:sequences}. In order to prove Theorem \ref{Theorem:main} we develop
in Section \ref{Section:turbulence} the theory of turbulence for Polish
groupoids. This is a generalization of Hjorth's theory of turbulence for
Polish group actions from \cite{hjorth_classification_2000}; see also \cite[%
Chapter 10]{gao_invariant_2009}.

We also study the (completely) isometric classification problem for
multiplier algebras $\mathcal{M}_V$ associated with operator algebraic
varieties in $\mathbb{B}_{d}$ with $d$ finite. In view of the results
mentioned above, this amounts to classifying operator algebraic varieties in 
$\mathbb{B}_{d}$ up to $\Aut(\mathbb{B}_d)$-conformal equivalence. We are
able to exactly determine the complexity of such a task.

\begin{theorem}
\label{Theorem:countable}For any $d\in \mathbb{N}$, the relation of $\Aut(%
\mathbb{B}_{d})$-conformal equivalence of operator algebraic varieties in $%
\mathbb{B}_{d}$ is essentially countable, and has maximum complexity among
essentially countable equivalence relations.
\end{theorem}

In particular, Theorem \ref{Theorem:countable} shows that the $\Aut(\mathbb{B%
}_{d})$-conformal equivalence classes of operator algebraic varieties in $%
\mathbb{B}_{d}$ cannot be explicitly parametrized by the points of a Polish
spaces. In other words the relation of $\Aut(\mathbb{B}_{d})$-conformal
equivalence of operator algebraic varieties in $\mathbb{B}_{d}$ is not
smooth. In fact, any class of complete invariants would have to be as
complex as conceivable. We will explain in more detail the content of
Theorem \ref{Theorem:countable} in Section \ref{Section:complexity}.

The rest of this article is organized as follows. In Section \ref%
{Section:complexity} we will recall the basic notions and results from Borel
complexity theory that we will refer to. Section \ref{Section:reproducing}
gives a short introduction to reproducing kernel Hilbert spaces and their
multiplier algebras. Here Theorem \ref{Theorem:main} is reduced to
establishing a nonclassification result for sequences in $(0,1]$ up to $\ell
^{\infty }$-equivalence. This is established in Section \ref%
{Section:sequences} by means of turbulence theory for Polish groupoids,
developed in Section \ref{Section:turbulence}. We conclude in Section \ref%
{Section:conformal} with the proof of Theorem \ref{Theorem:countable}.

\paragraph{Acknowledgments}

We would like to thank the anonymous referee for carefully reviewing the
paper and providing a large number of useful comments.

\section{Borel complexity theory\label{Section:complexity}}

\emph{Borel complexity theory} studies the relative complexity of
classification problems in mathematics, and offers tools to detect and prove
obstructions to classification. In this framework, a classification problem
is regarded as an \emph{equivalence relation} on a \emph{standard Borel space%
}. Perhaps after a suitable parametrization, this covers most of
classification problems in mathematics. For example, operator algebraic
varieties in $\mathbb{B}_{d}$ for $d\in \mathbb{N}\cup \left\{ \infty
\right\} $ are a collection $\mathcal{V}_{d}$ of nonempty closed subsets of $%
\mathbb{B}_{d}$. We will verify in the appendix that $\mathcal{V}_{d}$ is a
Borel subset of the space of nonempty closed subsets of $\mathbb{B}_{d}$
endowed with the Effros Borel structure \cite[Section 12.C]%
{kechris_classical_1995}. This shows that operator algebraic varieties form
a standard Borel space when endowed with the induced Borel structure \cite[%
Proposition 12.1]{kechris_classical_1995}. The relation of $\Aut(\mathbb{B}%
_{d})$-conformal equivalence of varieties in $\mathbb{B}_{d}$ can then be
regarded as an equivalence relation on this standard Borel space. Similarly,
the multiplier algebras $\mathcal{M}_{V}$ are naturally parametrized by the
varieties themselves, and one can regard algebraic isomorphisms of the
algebras $\mathcal{M}_{V}$ as an equivalence relation on the standard Borel
space of varieties described above.

Borel complexity theory aims at comparing the complexity of different
classification problems.\ The fundamental notion of comparison is Borel
reducibility. If $E$ and $F$ are equivalence relations on standard Borel
spaces $X$ and $Y$ respectively, then a \emph{Borel reduction} from $E$ to $%
F $ is a Borel function $f:X\rightarrow Y$ with the property that 
\begin{equation*}
f(x){}F{}f(x^{\prime })\quad \text{if and only if}\quad x{}E{}x^{\prime }
\end{equation*}%
for every $x,x^{\prime }\in X$. The relation $E$ is \emph{Borel reducible }%
to $F$---in formulas $E\leq _{B}F$---if there exists a Borel reduction from $%
E$ to $F$. This amounts to saying that one can assign to the elements of $X$
complete invariants up to $E$ that are $F$-equivalence classes, and moreover
such an assignment is \emph{constructive }in the sense that it is given by a
Borel map at the level of the spaces. We say that $E$ and $F$ are Borel
bireducible, and write $E\sim _{B}F$, if $E\leq _{B}F$ and $F\leq _{B}E$.
The notion of Borel reducibility was first introduced in \cite[Definition 2]%
{friedman_borel_1989}. A complete survey on Borel complexity theory can be
found in \cite{gao_invariant_2009}.

Some distinguished equivalence relations are used as benchmarks of
complexity to draw a hierarchy of classification problems is mathematics.
The first natural benchmark is provided by the relation $=_{\mathbb{R}}$ of
equality of real numbers. An equivalence relation is \emph{smooth }if it is
Borel reducible to $=_{\mathbb{R}}$. (One can replace $\mathbb{R}$ with any
other standard Borel space \cite[Theorem 15.6]{kechris_classical_1995}.) For
example, the relation of isomorphism of locally finite \emph{rooted} trees
is smooth \cite[Theorem 13.2.3]{gao_invariant_2009}.

Smooth equivalence relations represent the lowest level complexity. A more
ample class is given by considering Borel equivalence relations that are
countable or essentially countable. An equivalence relation $E$ on a
standard Borel space $X$ is \emph{Borel }if it is a Borel subset of the
product $X\times X$. A Borel equivalence relation $E$ is \emph{countable }if
its classes are countable, and \emph{essentially countable }if it is Borel
reducible to a countable one. Clearly, a smooth equivalence relation is, in
particular, essentially countable. The relation $E_{0}$ of tail equivalence
of binary sequences is countable but not smooth \cite[Subsection 6.1]%
{gao_invariant_2009}. More generally the orbit equivalence relation of a
Borel action of a countable group on a standard Borel space is countable.
There exists a countable Borel equivalence relation $E_{\infty }$ that has
maximum complexity among (essentially) countable Borel equivalence
relations. One can describe $E_{\infty }$ as the relation of isomorphism of
locally finite trees or graphs \cite[Theorem 13.2.4]{gao_invariant_2009}. In
the proof of Theorem \ref{Theorem:countable}, we will use the following
equivalent description of $E_{\infty }$. Let $F_{2}$ be the free group on
two generators and $\left\{ 0,1\right\} ^{F_{2}}$ the space of subsets of $%
F_{2}$ endowed with the product topology. The group $F_{2}$ naturally acts
on $\left\{ 0,1\right\} ^{F_{2}}$ by translation. The corresponding orbit
equivalence relation $E(F_{2},2)$ is Borel bireducible with $E_{\infty }$ 
\cite[Theorem 7.3.8]{gao_invariant_2009}.

A more generous notion of classifiability for equivalence relations is being 
\emph{classifiable by countable structures}. An equivalence relation is
classifiable by countable structures if it is Borel reducible to the
relation of isomorphism within some Borel class of structures in some first
order language. Equivalently an equivalence relation is classifiable by
countable structures if it is Borel reducible to the orbit equivalence
relation of a continuous action of $S_{\infty }$ on a Polish space \cite[%
Section 3.6]{gao_invariant_2009}. The Polish group $S_{\infty }$ is the
group of permutations of $\mathbb{N}$ with the topology of pointwise
convergence \cite[Section 2.4]{gao_invariant_2009}. Any (essentially)
countable equivalence relation is in particular classifiable by countable
structures \cite[Lemma 2.4, Lemma 2.5]{hjorth_complexity_2000}. Again, there
exists an equivalence relation of maximum complexity among those that are
classifiable by countable structures. Such an equivalence relation can be
described, for instance, as the relation of isomorphism of countable trees
or graphs \cite[Theorem 1]{friedman_borel_1989}.

\section{Kernels and multiplier algebras\label{Section:reproducing}}

A \emph{reproducing kernel Hilbert space} is a Hilbert space $\mathcal{H}$
which consists of functions on a set $X$ such that for every $w\in X$, the
functional 
\begin{equation*}
f\mapsto f(w)
\end{equation*}%
on $\mathcal{H}$ is bounded. Thus, there exists $K_{w}\in \mathcal{H}$ such
that 
\begin{equation*}
f(w)=\langle f,K_{w}\rangle _{\mathcal{H}}
\end{equation*}%
for all $f\in \mathcal{H}$. The function $K:X\times X\rightarrow \mathbb{C}$
defined by $K(z,w)=K_{w}(z)$ is called the reproducing kernel of $\mathcal{H}
$, and it is easy to check that $K$ is positive semi-definite in the sense
that for any $n\in \mathbb{N}$ and any collection of points $x_{1},\ldots
,x_{n}\in X$, the $n\times n$-matrix 
\begin{equation*}
\Big (K(x_{i},x_{j})\Big)_{i,j=1}^{n}
\end{equation*}%
is positive semi-definite. Conversely, if $K$ is a \emph{kernel} on $X$,
that is, $K:X\times X\rightarrow \mathbb{C}$ is positive semi-definite, then
there exists a unique reproducing kernel Hilbert space $\mathcal{H}(K)$ of
functions on $X$ such that $K$ is the reproducing kernel of $\mathcal{H}(K)$%
, see for example \cite[Section 2.2]{agler_pick_2002}. If $K$ is a kernel on 
$X$ and $Y\subset X$, then we denote by $K|_{Y}$ the kernel on $Y$ given by
the restriction of $K$ to $Y\times Y$. A kernel is \emph{irreducible }if for
any $z,w\in X$ one has that $K\left( z,w\right) \neq 0$, and furthermore if $%
z\neq w$ then $K_{z}$ and $K_{w}$ are linearly independent.

The \emph{multiplier algebra }of the reproducing kernel Hilbert space $%
\mathcal{H}$ is defined by 
\begin{equation*}
\Mult(\mathcal{H})=\{\varphi :X\rightarrow \mathbb{C}:\varphi \cdot f\in 
\mathcal{H}\text{ for all }f\in \mathcal{H}\}.
\end{equation*}%
It is a standard consequence of the closed graph theorem that for $\varphi
\in \Mult(\mathcal{H})$, the operator $M_{\varphi }$ on $\mathcal{H}$
defined by $M_{\varphi }f=\varphi \cdot f$ is bounded. We will always assume
that $1\in \mathcal{H}$, so that we may identify an element $\varphi $ of $%
\Mult(\mathcal{H})$ with its multiplication operator $M_{\varphi }$. This
identification endows $\Mult(\mathcal{H})$ with the structure of a
non-selfadjoint algebra of operators on $\mathcal{H}$. A good reference for
reproducing kernel Hilbert spaces and their multiplier algebras is the book 
\cite{agler_pick_2002}.

We will be interested in kernels with the \emph{complete Nevanlinna-Pick }%
(NP) \emph{property}. A kernel $K$ on a set $X$ is said to have the
Nevanlinna-Pick property if, given points $z_{1},\ldots ,z_{n}\in X$ and
complex numbers $\lambda _{1},\ldots ,\lambda _{n}\in \mathbb{C}$, the
existence of a multiplier $\varphi $ on $\mathcal{H}$ with $||M_{\varphi
}||\leq 1$ and 
\begin{equation*}
\varphi (z_{i})=\lambda _{i}\quad (i=1,\ldots ,n)
\end{equation*}%
is equivalent to positive semi-definiteness of the matrix 
\begin{equation*}
\Big(K(z_{i},z_{j})(1-\lambda _{i}\overline{\lambda _{j}})\Big)_{i,j=1}^{n}.
\end{equation*}%
The kernel is said to have the complete Nevanlinna-Pick property if the same
result holds for matrix-valued interpolation of arbitrary matrix size. More
information on the complete Nevanlinna-Pick property can be found in \cite[%
Section 5]{agler_pick_2002}. If $K$ satisfies the complete Nevanlinna-Pick
property, then $\mathcal{H}(K)$ is said to be a \emph{complete
Nevanlinna-Pick space}. We say that $\mathcal{H}(K)$ is irreducible if $K$
is irreducible.

Perhaps the most important example of a complete Nevanlinna-Pick space is
the \emph{Drury-Arveson space}. For $d\in \mathbb{N}\cup \left\{ \infty
\right\} $, we denote the Drury-Arveson kernel on $\mathbb{B}_{d}$ by 
\begin{equation*}
A_{d}\left( z,w\right) =\frac{1}{1-\left\langle z,w\right\rangle }.
\end{equation*}%
This kernel has the complete NP property, and $H_{d}^{2}$ is the reproducing
kernel Hilbert space associated with the space $\mathcal{H}(A_{d})$ (see 
\cite[Section 8.1]{agler_pick_2002}). If $V\subset \mathbb{B}_{d}$ then the
restriction $A_{d}|_{V}$ of the Drury-Arveson kernel is a kernel on $V$. The
associated complete NP space $H_{d}^{2}|_{V}$ is the closed linear span of
functions on $V$ of the form $z\mapsto \frac{1}{1-\left\langle
z,w\right\rangle }$ for $w\in V$. It is a consequence of the Nevanlinna-Pick
property that every multiplier of $H_{d}^{2}|_{V}$ is the restriction of a
multiplier of $H_{d}^{2}$ to $V$, thus the multiplier algebra $\Mult%
(H_{d}^{2}|_{V})$ coincides with the algebra $\mathcal{M}_{V}$ defined in
the introduction. In particular, this endows $\mathcal{M}_{V}$ with the
structure of a non-selfadjoint operator algebra. According to a theorem of
Agler and McCarthy, $H_{\infty }^{2}$ is universal in the sense that every
separable irreducible complete Nevanlinna-Pick space can be identified with $%
H_{\infty }^{2}|_{V}$ for some operator algebraic variety $V\subset \mathbb{B%
}_{\infty }$, see \cite[Section 8]{agler_pick_2002}. The article \cite%
{shalit_operator_2013} offers a comprehensive survey on the Drury-Arveson
space and its properties.

We will observe in the appendix that the collection $\mathcal{V}_{\infty }$
of operator algebraic varieties $V\subset \mathbb{B}_{\infty }$ is a Borel
subset of the space of nonempty closed subspaces of $\mathbb{B}_{\infty }$
endowed with the Effros Borel structure \cite[Section 12.C]%
{kechris_classical_1995}. Therefore we can state more precisely Theorem \ref%
{Theorem:main} as follows.

\begin{theorem}
\label{Theorem:main2}The equivalence relation on the space $\mathcal{V}%
_{\infty }$ of operator algebraic varieties in $\mathcal{7}\mathbb{B}%
_{\infty }$ defined by $V\sim W$ if and only if $\mathcal{M}_{V}$ and $%
\mathcal{M}_{W}$ are algebraically isomorphic is not classifiable by
countable structures.
\end{theorem}

An important class of irreducible complete NP kernels are the unitarily
invariant complete NP kernels on $\mathbb{B}_{d}$. By definition, these
kernels admit a power series representation 
\begin{equation*}
K_{\boldsymbol{a}}^{(d)}(z,w)=\sum_{n=0}^{\infty }a_{n}\langle z,w\rangle
^{n}
\end{equation*}%
for $z,w\in \mathbb{B}_{d}$, where $\boldsymbol{a}=(a_{n})$ is a sequence of
positive numbers such that $a_{0}=1$, the power series $\sum_{n=0}^{\infty
}a_{n}t^{n}$ has radius of convergence $1$, $\sum_{n=0}^{\infty
}a_{n}=\infty $, and there exists a sequence $\boldsymbol{b}=(b_{n})$ of
non-negative numbers such that 
\begin{equation}
\sum_{n=0}^{\infty }a_{n}t^{n}=\frac{1}{1-\sum_{n=1}^{\infty }b_{n}t^{n}},
\label{eqn:NP}
\end{equation}%
see \cite[Section 7]{hartz_isomorphism_2015}. We let $\mathcal{A}\subset
(0,\infty )^{\mathbb{N}}$ denote the set of such sequences. It is not
difficult to see that the set $\mathcal{B}$ of pairs $\left( \boldsymbol{a},%
\boldsymbol{b}\right) $ of sequences satisfying \eqref{eqn:NP} is a Borel
subset of $\left( 0,\infty \right) ^{\mathbb{N}}\times \left[ 0,\infty
\right) ^{\mathbb{N}}$. From this and the fact that every section of $%
\mathcal{B}$ has cardinality at most $1$, it follows that $\mathcal{A}$ is
the range of an injective Borel map and hence Borel by \cite[Corollary 15.2]%
{kechris_classical_1995}.

It follows from the universality result of Agler and McCarthy mentioned
above that for every $\boldsymbol{a}\in \mathcal{A}$, there exists a variety 
$V_{\boldsymbol{a}}\subset \mathbb{B}_{\infty }$ such that $\mathcal{M}_{V_{%
\boldsymbol{a}}}$ is completely isometrically isomorphic to $\Mult(\mathcal{H%
}(K_{\boldsymbol{a}}^{(d)}))$. Furthermore the variety $V_{\boldsymbol{a}}$
can be taken as the image of $\mathbb{B}_{d}$ under an embedding $j_{%
\boldsymbol{a}}:\mathbb{B}_{d}\rightarrow \mathbb{B}_{\infty }$. The
discussion preceding Proposition 11.8 in \cite{hartz_isomorphism_2015} shows
that the embedding $j_{\boldsymbol{a}}$ can be explicitly defined in terms
of $\boldsymbol{a}$. One can use this to show that there exists a Borel map $%
\boldsymbol{a}\mapsto V_{\boldsymbol{a}}$ from $\mathcal{A}$ to $\mathcal{V}%
_{\infty }$ such that $\mathcal{M}_{V_{\boldsymbol{a}}}$ is completely
isometrically isomorphic to $\Mult(H(K_{\boldsymbol{a}}^{(d)}))$. Therefore,
in order to establish Theorem \ref{Theorem:main2}, it is enough to prove the
following result.

\begin{theorem}
\label{Theorem:nonclassification} Let $d\in \mathbb{N}$. The relation $\sim
_{d}$ on the space $\mathcal{A}$ defined by $\boldsymbol{a}\sim _{d}%
\boldsymbol{a}^{\prime }$ if and only if $\Mult(H(K_{\boldsymbol{a}}^{(d)}))$
and $\Mult(H(K_{\boldsymbol{a}\mathbf{^{\prime }}}^{(d)}))$ are
algebraically isomorphic is not classifiable by countable structures.
\end{theorem}

To prove this result, we will consider a special class of unitarily
invariant complete NP kernels. We say that a (necessarily nonincreasing)
sequence $\boldsymbol{a}=\left( a_{n}\right) $ in $(0,1]^{\mathbb{N}}$ is 
\emph{admissible log-convex} if $a_{0}=1$, $(\frac{a_{n}}{a_{n+1}})_{n}$ is
nonincreasing and converges to $1$, and $\sum_{n=0}^{\infty }a_{n}=\infty $.
Let $\mathcal{A}_{0}\subset (0,1]^{\mathbb{N}}$ be the Borel set of
admissible log-convex sequences. Log-convexity of $\boldsymbol{a}$ implies
that there exists a sequence $(b_{n})$ of non-negative numbers as in
Equation \eqref{eqn:NP}; see \cite[Lemma 7.38]{agler_pick_2002}. Therefore, $%
\mathcal{A}_{0}\subset \mathcal{A}$.

We consider on $\mathcal{A}_{0}$ the relation $E_{\mathcal{A}_{0}}$ defined
by $\boldsymbol{a}E_{\mathcal{A}_{0}}\boldsymbol{a}^{\prime }$ if and only
if $\boldsymbol{a}$ and $\boldsymbol{a}^{\prime }$ have the same growth or,
using Landau's notation, $\boldsymbol{a}=\Theta \left( \boldsymbol{a}%
^{\prime }\right) $. This means that there are constants $c,C>0$ such that $%
c\leq a_{n}^{\prime }/a_{n}\leq C$ for every $n\in \mathbb{N}$. The
equivalence of (ii) and (iii) in \cite[Corollary 11.7]%
{hartz_isomorphism_2015} shows that the relations $\sim _{d}$ and $E_{%
\mathcal{A}_{0}}$ coincide on $\mathcal{A}_{0}$. Therefore, it only remains
to show that the relation $E_{\mathcal{A}_{0}}$ is not classifiable by
countable structures. This will be proved in Section \ref{Section:sequences}.

We mention here that the same proof also shows that the algebras $A(K)$ for $%
K$ a unitarily invariant complete NP kernel on $\mathbb{B}_{d}$ are not
classifiable by countable structures up to algebraic isomorphism. Here $A(K)$
denotes the closure of the polynomials in $\mathrm{Mult}(\mathcal{H}(K))$,
see \cite[Section 6]{hartz_isomorphism_2015}. One can also observe that, for 
$d\in \mathbb{N}$, the collection $\mathcal{K}$ of unitarily invariant
complete NP kernels is Borel. It follows from Theorem \ref%
{Theorem:nonclassification} that the relation on $\mathcal{K}$ defined by $%
K\sim K^{\prime }$ if and only if $\mathrm{Mult}(\mathcal{H}(K))$ and $%
\mathrm{Mult}(\mathcal{H}(K'))$ are algebraically isomorphic is not
classifiable by countable structures.

\section{Turbulence for Polish groupoids\label{Section:turbulence}}

The main goal of this section is to introduce the notion of \emph{turbulence}
for Polish groupoids, and to generalize to this setting Hjorth's turbulence
theorem. A \emph{groupoid} can be seen a tuple $(G,G^{0},s,r,\cdot ,i,\left(
\quad \right) ^{-1})$ where $G^{0}$ and $G$ are sets, $s,r$ are functions $%
G\rightarrow G^{0}$, $i$ is a function $G^{0}\rightarrow G$, $\cdot $ is a
function from the set $G^{2}:=\left\{ \left( \gamma ,\rho \right) \in
G\times G:s(\gamma )=r(\rho )\right\} $ to $G$, and $\gamma \mapsto \gamma
^{-1}$ is a function from $G$ to $G$. These functions are assumed to satisfy
the following relations:

\begin{itemize}
\item $s\left( i(x)\right) =r\left( i(x)\right) =x$ for every $x\in G^{0}$;

\item $\gamma \cdot i(s(\gamma ))=\gamma =i\left( r(\gamma )\right) \cdot
\gamma $ for every $\gamma \in G$;

\item $s\left( \gamma \cdot \rho \right) =s(\rho )$ and $r\left( \gamma
\cdot \rho \right) =r(\gamma )$ for every $\gamma ,\rho \in G$;

\item $\left( \gamma \cdot \rho \right) \cdot \tau =\gamma \cdot \left( \rho
\cdot \tau \right) $ for every $\gamma ,\rho ,\tau \in G$ such that $\left(
\gamma ,\rho \right) \in G^{2}$ and $\left( \rho ,\tau \right) \in G^{2}$;

\item $\gamma ^{-1}\cdot \gamma =s(\gamma )$ and $\gamma \cdot \gamma
^{-1}=r(\gamma )$ for every $\gamma \in G$.
\end{itemize}

The elements of $G^{0}$ are called the \emph{objects }of the groupoid, while
the elements of $G$ are called the \emph{arrows}. Given an arrow $\gamma $
the objects $s(\gamma )$ and $r(\gamma )$ are called the source and range of 
$\gamma $, respectively, while $\gamma ^{-1}$ is called the inverse arrow of 
$\gamma $. The arrow $i(x)$ associated with $x$ is called the \emph{identity
arrow }of $X$. The maps $s,r:G\rightarrow G^{0}$ are called source and range
maps, respectively, while the partially defined binary operation $\cdot $ is
called composition of arrows. As it is customary, in the following we
identify every object with the corresponding identity arrow, and we denote
the composition of arrows $\gamma \cdot \rho $ simply by $\gamma \rho $. A
groupoid can be equivalently defined as a small category where every
morphism is an isomorphism. In this case, the arrows of the groupoid are
just the morphisms of the category.

A\emph{\ Polish groupoid} is a groupoid endowed with a topology that

\begin{enumerate}
\item has a countable basis of Polish open sets,

\item makes composition and inversion of arrows continuous and open,

\item makes for every $x\in G^{0}$ the set $Gx$ of arrows of $G$ with source 
$x$ a Polish subspace of $G$, and

\item makes the set of objects $G^{0}$ a Polish subspace of $G$.
\end{enumerate}

Polish groupoids have been introduced and studied in \cite%
{ramsay_mackey-glimm_1990,ramsay_polish_2000}. In \cite{lupini_polish_2014}
several fundamental results about Polish group actions are generalized to
Polish groupoids. We assume in the following that $G$ is a Polish groupoid.
The \emph{orbit equivalence relation }of $G$ is the equivalence relation $%
E_{G}$ on $G^{0}$ defined by $xE_{G}y$ if and only if there exists $\gamma
\in G$ such that $s(\gamma )=x$ and $r(\gamma )=y$. If $A,B\subset G$ we let 
$AB$ be the set of all compositions $\gamma \rho $ for $\gamma \in A$ and $%
\rho \in B$ such that $r(\rho )=s(\gamma )$. We write $A\gamma $ for $%
A\left\{ \gamma \right\} $ when $A\subset G$ and $\gamma \in G$. In
particular if $x\in G^{0}$ then $Ax$ is the set of elements of $A$ with
source $x$. If $X$ is a $G_{\delta }$ subset of $G^{0}$, denote by $G|_{X}$
the Polish groupoid $XGX=\left\{ \gamma \in G:s(\gamma ),r(\gamma )\in
X\right\} $ endowed with the subspace topology. This is called the \emph{%
restriction }of $G$ to $X$. If $x$ is an object of $G$ and $V$ is a
neighborhood of $x$ in $G$, then the \emph{local orbit }$\mathcal{O}\left(
x,V\right) $ is the set of all points that can be reached from $x$ by
applying elements of $V$. In formulas%
\begin{equation*}
\mathcal{O}\left( x,V\right) =\bigcup_{n\in \mathbb{N}}r\left[ V^{n}x\right] 
\text{.}
\end{equation*}

\begin{definition}
\label{Definition:turbulence}An object $x$ of $G$ is \emph{turbulent }if for
every neighborhood $V$ of $x$ the local orbit $\mathcal{O}\left( x,V\right) $
is somewhere dense. The groupoid $G$ is \emph{generically preturbulent} if
the set of turbulent objects with dense orbit is a comeager subset of $G^{0}$%
. If moreover every orbit is meager, then $G$ is \emph{generically turbulent}%
.
\end{definition}

In the rest of this section we will often tacitly use the following version
of the classical Kuratowski--Ulam theorem; see \cite[Lemma 2.9.1]%
{lupini_polish_2014}.

\begin{fact}
\label{Fact:KU}Suppose that $X$ is a second countable topological space, $Y$
is a Polish space, and $f:X\rightarrow Y$ is open and continuous. If $%
A\subset X$ is analytic, then $A$ is comeager if and only if $f^{-1}\left\{
y\right\} \cap A$ is comeager in $f^{-1}\left\{ y\right\} $ for comeager
many $y\in Y$.
\end{fact}

For example, it follows from Fact \ref{Fact:KU} that if $X$ is a dense $%
G_{\delta }$ subspace of $G^{0}$ and $G$ is generically (pre)turbulent, then 
$G|_{X}$ is generically (pre)turbulent.

Suppose that $H$ is a Polish group and $Y$ is a Polish $H$-space, i.e.\ a
Polish space endowed with a continuous action of $H$. Let $G$ be the Polish
action groupoid associated with the Polish $H$-space $Y$ as in \cite[%
Subsection 2.7]{lupini_polish_2014}. Observe that the orbit equivalence
relation $E_{G}$ coincides with the orbit equivalence relation $E_{H}^{Y}$.
Furthermore it is not difficult to verify that $G$ is a generically
(pre)turbulent groupoid as in Definition \ref{Definition:turbulence} if and
only if $Y$ is a generically (pre)turbulent $H$-space in the sense of \cite[%
Definition 10.3.3]{gao_invariant_2009}.

Recall the following terminology from Borel complexity of equivalence
relations. If $E$ and $F$ are equivalence relations on standard Borel spaces 
$X$ and $Y$, then an $\left( E,F\right) $-homomorphism is a function $%
f:X\rightarrow Y$ that maps $E$-classes into $F$-classes. A generic $\left(
E,F\right) $-homomorphism is a function $f:X\rightarrow Y$ that is an $%
\left( E,F\right) $-homomorphism when restricted to some comeager subset of $%
X$. An equivalence relation $E$ on a standard Borel space $X$ is \emph{%
generically }$S_{\infty }$\emph{-ergodic} if for every Polish $S_{\infty }$%
-space $Y$ and every Baire-measurable generic $(E,E_{S_{\infty }}^{Y})$%
-homomorphism, there exists a comeager subset of $X$ that is mapped by $f$
into a single $S_{\infty }$-orbit. It is well known that an equivalence
relation is classifiable by countable structures if and only if it is Borel
reducible to the orbit equivalence relation of a Polish $S_{\infty }$-space;
see \cite[Theorem 11.3.8]{gao_invariant_2009}.

The following is the main consequence of turbulence for Polish groupoids.

\begin{theorem}
\label{Theorem:turbulence-groupoids}Suppose that $G$ is a generically
preturbulent Polish groupoid. Then the associated orbit equivalence relation 
$E_{G}$ is generically $S_{\infty }$-ergodic.
\end{theorem}

\begin{corollary}
\label{Corollary:turbulence-groupoids}If $G$ is a generically turbulent
Polish groupoid, then the orbit equivalence relation $E_{G}$ is not
classifiable by countable structures.
\end{corollary}

Theorem \ref{Theorem:turbulence-groupoids} generalizes the original result
of Hjorth \cite[Section 3]{hjorth_classification_2000} from Polish group
actions to Polish groupoids. Polish groupoids provide a natural setting to
present the proof of Hjorth's turbulence theorem even in the case of Polish
group actions. Indeed in the course of the proof one looks at the action
\textquotedblleft restricted\textquotedblright\ to a (not necessarily
invariant) $G_{\delta }$ subspace; see for example \cite[Theorem 10.4.2]%
{gao_invariant_2009}.\ Such a restriction is not a Polish group action in
general, even when one starts with a Polish group action. It is nonetheless
a Polish groupoid.

The following lemma is the groupoid analog of \cite[Lemma 3.17]%
{hjorth_classification_2000}. In the following, if $V$ is an open subset of $%
G$, we write $\forall ^{\ast }\gamma \in V$ to mean \textquotedblleft for a
comeager set of $\gamma \in V$\textquotedblright .

\begin{lemma}
\label{Lemma:neighborhood-origin}Suppose that $G$ is a Polish groupoid, $H$
is a Polish group, and $Y$ is a Polish $H$-space. If $f:G^{0}\rightarrow Y$
is a Baire-measurable generic $\left( E_{G},E_{H}^{Y}\right) $-homomorphism,
then there exists a comeager subset $C$ of $G^{0}$ such that for every $x\in
C$ and every open neighborhood $W$ of $1_{H}$ in $H$ there exists a
neighborhood $V$ of $x$ such that for every $x^{\prime }\in s[V]\cap C$ and
for a comeager set of $\gamma \in Vx^{\prime }$, 
\begin{equation*}
f(r(\gamma ))\in Wf(x^{\prime })\text{.}
\end{equation*}
\end{lemma}

\begin{proof}
After replacing $G$ with the restriction of $G$ to a dense $G_{\delta }$
subset of $G^{0}$, we can assume that $f$ is a continuous $\left(
E_{G},E_{H}^{Y}\right) $-homomorphism \cite[Exercise 2.3.2]%
{gao_invariant_2009}. Furthermore it is enough to prove that for every open
neighborhood $W$ of $1_{H}$ there is a comeager subset $C$ of $X$ such that
for every $x\in C$ there exists a neighborhood $V$ of $x$ in $G$ such that $%
\forall x^{\prime }\in s[V]\cap C$, $\forall ^{\ast }\gamma \in Vx^{\prime }$%
, $f(r(\gamma ))\in Wf(x^{\prime })$. Fix an open neighborhood $W$ of $1_{H}$
and an open neighborhood $W_{0}$ of $1_{H}$ such that $W_{0}^{-1}=W_{0}$ and 
$W_{0}^{2}\subset W$. Fix a sequence $\left( h_{n}\right) $ in $H$ such that%
\begin{equation*}
\bigcup_{n\in \mathbb{N}}W_{0}h_{n}=H\text{.}
\end{equation*}%
For every $n\in \mathbb{N}$, the set%
\begin{equation*}
B_{n}=\left\{ \left( z,y\right) \in Y\times Y\mid z\in W_{0}h_{n}y\right\}
\end{equation*}%
is analytic. Therefore the set 
\begin{equation*}
A_{n}=\left\{ \gamma \in G:f(r(\gamma ))\in W_{0}h_{n}f(s(\gamma ))\right\}
\end{equation*}%
is analytic by \cite[Proposition 22.1]{kechris_classical_1995}. By \cite[%
Proposition 8.22]{kechris_classical_1995} there exists an open subset $O_{n}$
of $G$ such that $O_{n}\bigtriangleup A_{n}$ is meager. Set $D_{n}=A_{n}\cap
O_{n}$, and observe that $D_{n}D_{n}^{-1}$ is a comeager subset of $%
O_{n}O_{n}^{-1}$. Since $G$ is the union of $A_{n}$ for $n\in \mathbb{N}$,
the union $O$ of $O_{n}$ for $n\in \mathbb{N}$ is an open dense subset of $G$%
. In particular $r\left[ O\right] $ is an open subset of $G^{0}$. Define
now, for $n\in \mathbb{N}$, $\widetilde{O}_{n}$ to be the set of $\gamma \in
O_{n}$ such that $r(\gamma )$ does not belong to the closure of the union of 
$r\left[ O_{i}\right] $ for $i<n$. Let $\widetilde{O}$ be the union of $%
\widetilde{O}_{n}$ for $n\in \mathbb{N}$, and observe that $r[\widetilde{O}]$
is an open dense subset of $G^{0}$. For every $n\in \mathbb{N}$ set $%
\widetilde{D}_{n}=D_{n}\cap \widetilde{O}_{n}$ and observe that $\widetilde{D%
}_{n}$ is a comeager subset of $\widetilde{O}_{n}$. Therefore there exists a
comeager subset $C_{n}$ of $r[\widetilde{O}_{n}]=s[\widetilde{O}_{n}%
\widetilde{O}_{n}^{-1}]$ such that for every $x\in C_{n}$, $\widetilde{D}_{n}%
\widetilde{D}_{n}^{-1}x$ is a comeager subset of $\widetilde{O}_{n}%
\widetilde{O}_{n}^{-1}x$. Define $C$ to be the union of $C_{n}$ for $n\in 
\mathbb{N}$, and observe that $C$ is a comeager subset of $G^{0}$. We claim
that $C$ satisfies the desired conclusions. Fix $x\in C$ and $n\in \mathbb{N}
$ such that $x\in C_{n}$. We have that $\widetilde{O}_{n}\widetilde{O}%
_{n}^{-1}$ is an open neighborhood of $x$. Furthermore for every $x^{\prime
}\in C_{n}=C\cap s[\widetilde{O}_{n}\widetilde{O}_{n}^{-1}]$, $\widetilde{D}%
_{n}\widetilde{D}_{n}^{-1}x^{\prime }$ is comeager in $\widetilde{O}_{n}%
\widetilde{O}_{n}^{-1}x^{\prime }$. If $\rho ,\gamma \in \widetilde{D}_{n}$,
then%
\begin{equation*}
f(r(\gamma ))\in W_{0}h_{n}f(s(\gamma ))\quad \text{and\quad }f(r(\rho ))\in
W_{0}h_{n}f(s(\rho ))\text{.}
\end{equation*}%
Therefore%
\begin{equation*}
f(r(\rho \gamma ^{-1}))=f(r(\rho ))\in W_{0}h_{n}f(s(\rho ))\subset
W_{0}W_{0}^{-1}f(s(\rho \gamma ^{-1}))\subset Wf(s(\rho \gamma ^{-1}))\text{.%
}
\end{equation*}%
This concludes the proof.
\end{proof}

We now explain how one can deduce Theorem \ref{Theorem:turbulence-groupoids}
from Lemma \ref{Lemma:neighborhood-origin}.

\begin{proof}[Proof of Theorem \protect\ref{Theorem:turbulence-groupoids}]
Fix an enumeration $\left( V_{k}\right) _{k\in \mathbb{N}}$ of a basis of
Polish open subsets of $G$, and a compatible complete metric $d_{Y}$ on $Y$
bounded by $1$. Suppose that $d$ is the metric in $S_{\infty }$ defined by 
\begin{equation*}
\log _{2}d(\sigma ,\rho )=-\min \left\{ n\in \mathbb{N}:\sigma \left(
n\right) \neq \rho \left( n\right) \right\} \text{.}
\end{equation*}%
for $\sigma ,\rho \in S_{\infty }$. We also consider the complete metric%
\begin{equation*}
D(\sigma ,\rho )=d(\sigma ,\rho )+d(\sigma ^{-1},\rho ^{-1})
\end{equation*}%
on $S_{\infty }$. Define $e$ to be the identity of $S_{\infty }$, and%
\begin{equation*}
N_{k}=\left\{ \sigma \in S_{\infty }:d(\sigma ,e)<2^{-k}\right\}
\end{equation*}%
for $k\in \mathbb{N}$. As in the proof of Hjorth's turbulence theorem for
Polish group actions \cite[Theorem 10.4.2]{gao_invariant_2009}, one can
deduce from Lemma \ref{Lemma:neighborhood-origin} that there exists a dense $%
G_{\delta }$ subset $C_{0}$ of $G^{0}$ with the following properties:

\begin{itemize}
\item $f|_{C_{0}}$ is a continuous $(E_{G},E_{S_{\infty }}^{Y})$%
-homomorphism;

\item every element of $C_{0}$ has dense orbit;

\item for every $m\in \mathbb{N}$ and $x\in V_{m}\cap C_{0}$ the local orbit 
$\mathcal{O}\left( x,V_{m}\right) $ is somewhere dense;

\item for every $x\in C_{0}$ and $k\in \mathbb{N}$ there exists $m\in 
\mathbb{N}$ such that $x\in V_{m}$ and $\forall x^{\prime }\in s\left[ V_{m}%
\right] $, $\forall ^{\ast }\gamma \in V_{m}x^{\prime }$, $f(r(\gamma ))\in
N_{k}f(x^{\prime })$.
\end{itemize}

Let $C$ be the set of $x\in C_{0}$ such that $\forall ^{\ast }\gamma \in Gx$%
, $r(\gamma )\in C_{0}$, and observe that $C$ is a dense $G_{\delta }$
subset of $G^{0}$ \cite[Lemma 2.10.6]{lupini_polish_2014}. After replacing $%
G $ with the restriction $G|_{C}$ of $G$ to $C$, and $V_{k}$ with $V_{k}\cap
G|_{C}$, we can assume that $C=G^{0}$.

Fix $x_{0},y_{0}\in G^{0}$. We claim that $f(x)E_{S_{\infty }}^{Y}f(y)$. We
will define by recursion on $i\geq 0$ elements $x_{i},y_{i}$ of $G^{0}$, $%
g_{i},h_{i}$ of $S_{\infty }$, and $n_{x}(i),n_{y}(i)$ of $\mathbb{N}$, such
that the following conditions hold:

\begin{itemize}
\item $g_{0}=h_{0}=e$;

\item $x_{0}\in V_{n_{x}\left( 0\right) }$ and $y_{0}\in V_{n_{y}\left(
0\right) }$;

\item $g_{i}f(x)=f(x_{i})$ and $h_{i}f(y)=f(y_{i})$;

\item $x_{i+1}\in V_{n_{x}(i)}\cap \mathcal{O}\left(
x_{i},V_{n_{x}(i)}\right) $ and $y_{i+1}\in V_{n_{y}(i)}\cap \mathcal{O}%
\left( y_{i},V_{n_{y}(i)}\right) $;

\item the $d_{Y}$-diameter of $f\left[ G^{0}\cap V_{n_{x}(i)}\right] $ is at
most $2^{-i}$;

\item $\mathcal{O}\left( x_{i},V_{n_{x}(i)}\right) $ is dense in $%
V_{n_{y}(i)}\cap G^{0}$ and $\mathcal{O}\left( y_{i},V_{n_{y}(i)}\right) $
is dense in $V_{n_{x}(i+1)}\cap G^{0}$;

\item $d(g_{i},g_{i+1})\leq 2^{-i}$ and $d(h_{i},h_{i+1})\leq 2^{-i}$;

\item if $i>0$ and $k_{x}(i)=\max \left\{ g_{i}(\lambda ),g_{i}^{-1}(\lambda
)\mid \lambda \leq i\right\} $, then $\forall z\in s\left[ V_{n_{x}\left(
i\right) }\right] $, $\forall ^{\ast }\gamma \in V_{n_{x}(i)}z$, $f(r(\gamma
))\in N_{k_{x}(i)}f(z)$;

\item if $i\geq 0$ and $k_{y}(i)=\max \left\{ h_{i}(\lambda
),h_{i}^{-1}(\lambda )\mid \lambda \leq i\right\} $, then $\forall z\in s%
\left[ V_{n_{y}\left( i\right) }\right] $, $\forall ^{\ast }\gamma \in
V_{n_{y}(i)}z$, $f(r(\gamma ))\in N_{k_{y}(i)}f(z)$.
\end{itemize}

Granted the construction, the sequences $\left( g_{i}\right) ,\left(
h_{i}\right) $ in $S_{\infty }$ are $D$-Cauchy and hence converge to
elements $g,h\in S_{\infty }$. Furthermore $d_{Y}\left(
g_{i}f(x),h_{i}f(y)\right) \rightarrow 0$ and hence $gf(x)=hf(y)$. This
concludes the proof that $f(x)E_{S_{\infty }}^{Y}f(y)$.

We assume recursively that we have defined $%
x_{i},y_{i},g_{i},h_{i},n_{x}(i),n_{y}(i)$ and explain how to define $%
x_{i+1},g_{i+1},n_{x}(i+1)$. The definition of $y_{i+1},h_{i+1},n_{y}(i+1)$
is similar. We have that the local orbit $\mathcal{O}\left(
y_{i},V_{n_{y}(i)}\right) $ is somewhere dense. Pick a nonempty open subset $%
W$ of $V_{n_{y}(i)}$ that is contained in the closure of $\mathcal{O}\left(
y_{i},V_{n_{y}(i)}\right) $. By recursive hypothesis we have that $\mathcal{O%
}\left( x_{i},V_{n_{x}(i)}\right) $ is dense in $W$. Let $\gamma _{0},\ldots
,\gamma _{\ell -1}\in V_{n_{x}(i)}$ such that, setting $z_{j}=s\left( \gamma
_{j}\right) $ for $j<\ell $ and $z_{\ell }=r\left( \gamma _{\ell -1}\right) $%
, one has that $z_{0}=x_{i}$, $z_{\ell }\in W$, and $z_{j+1}=r\left( \gamma
_{j}\right) $ for $j<\ell $. Since by inductive assumption we have that $%
\forall z\in s\left[ V_{n_{x}(i)}\right] $, $\forall ^{\ast }\gamma \in
V_{n_{x}(i)}z$, $f(r(\gamma ))\in N_{k_{x}(i)}f(z)$, after modifying the
sequence $\left( \gamma _{0},\ldots ,\gamma _{\ell -1}\right) $ we can
assume that, for every $j<\ell $, $f\left( z_{j+1}\right) =p_{j}f\left(
z_{j}\right) $ for some $p_{j}\in N_{k_{x}\left( i\right) }$. Therefore $%
f(z_{\ell })=pf(z)$ where $p=p_{\ell -1}p_{\ell -2}\cdots p_{0}\in
N_{k_{x}(i)}$. We may then let $x_{i+1}=z_{\ell }$, $g_{i+1}=pg_{i}$, $%
k_{x}(i+1)=\max \left\{ g_{i+1}(\lambda ),g_{i+1}^{-1}(\lambda ):\lambda
\leq i+1\right\} $, and $n_{x}(i+1)\in \mathbb{N}$ such that $x_{i+1}\in
V_{n_{x}(i+1)}$ and $\forall x^{\prime }\in s\left[ V_{n_{x}\left(
i+1\right) }\right] $, $\forall ^{\ast }\gamma \in V_{n_{x}\left( i+1\right)
}x^{\prime }$, $f(r(\gamma ))\in N_{k_{x}(i+1)}f(x^{\prime })$. This
concludes the definition of $x_{i+1},g_{i+1},n_{x}(i+1)$.
\end{proof}

\section{Admissible log-convex sequences\label{Section:sequences}}

Recall from Section \ref{Section:reproducing} that a sequence $\boldsymbol{a}
$ in $(0,1]^{\mathbb{N}}$ is \emph{admissible log-convex} if $a_{0}=1$, $(%
\frac{a_{n}}{a_{n+1}})_{n}$ is nonincreasing and converges to $1$, and $%
\sum_{n}a_{n}=\infty $. The set $\mathcal{A}_{0}\subset (0,1]^{\mathbb{N}}$
of admissible log-convex sequences is Borel. We say that two admissible
log-convex sequences $\boldsymbol{a}$ and $\boldsymbol{a}^{\prime }$ are 
\emph{growth equivalent }if there exist $c,C>0$ such that $c\leq \frac{%
a_{n}^{\prime }}{a_{n}}\leq C^{\prime }$ for every $n\in \mathbb{N}$. We
denote by $E_{\mathcal{A}_{0}}$ the corresponding equivalence relation on $%
\mathcal{A}_{0}$. The main goal of this section is to prove the following
result:

\begin{proposition}
\label{Proposition:nonclassify-sequences}Admissible log-convex sequences
considered up to growth equivalence are not classifiable by countable
structures.
\end{proposition}

However, it is not difficult to verify that admissible log-convex sequences
are classifiable by the orbits of a Polish group action up
to growth equivalence.
This means that there exists a continuous Polish
group action $G\curvearrowright X$ such that $E_{\mathcal{A}_{0}}$ is Borel
reducible to the orbit equivalence relation $E_{G}^{X}$. The crucial point
is that if 
\begin{equation*}
B=\{(-\log (a_{n})):(a_{n})\in A\}\subset (0,\infty )^{\mathbb{N}},
\end{equation*}%
then 
\begin{equation*}
H=\{z\in \ell ^{\infty }:\text{ there exist }x,y\in B\text{ with }x-y=z\}
\end{equation*}%
is a subgroup of $\ell ^{\infty }$ which is separable in the $\ell ^{\infty
} $-metric, and two sequences $(a_{n}),(a_{n}^{\prime })$ in $A$ have the
same growth if and only if $(-\log (a_{n}))$ and $(-\log (a_{n}^{\prime }))$
belong to the same $H$-orbit under translation.

The rest of this section is dedicated to the proof of Proposition \ref%
{Proposition:nonclassify-sequences}. Consider the equivalence relation $F$
on $\left( 0,1\right) ^{\mathbb{N}}$ defined by%
\begin{equation*}
\boldsymbol{s}F\boldsymbol{s}^{\prime }\text{ if and only if }%
\sup_{n}\left\vert \sum_{k<n}\left( \prod_{i\leq k}s_{i}-\prod_{i\leq
k}s_{i}^{\prime }\right) \right\vert < \infty \text{.}
\end{equation*}%
Define furthermore the Borel function%
\begin{eqnarray*}
(0,1)^{\mathbb{N}} &\rightarrow &(0,1]^{\mathbb{N}} \\
\boldsymbol{s} &\mapsto &f\left( \boldsymbol{s}\right) =\exp \left(
-\sum_{k<n}\prod_{i\leq k}s_{i}\right) _{n\in \mathbb{N}}
\end{eqnarray*}%
where the empty sum is $0$. Observe that for $s\in (0,1)^{\mathbb{N}}$, we
have that $f(\boldsymbol{s})_{0}=1$, $f(\boldsymbol{s})$ is log-convex and $%
f(\boldsymbol{s})_{n}/f(\boldsymbol{s})_{n+1}\geq 1$ for all $n\in \mathbb{N}
$. Let $X\subset \left( 0,1\right) ^{\mathbb{N}}$ be the set of $\boldsymbol{%
s}\in \left( 0,1\right) ^{\mathbb{N}}$ such that $f\left( \boldsymbol{s}%
\right) \in \mathcal{A}_{0}$.\ Using the fact that $f(\boldsymbol{s})\in 
\mathcal{A}_{0}$ if and only if $f(\boldsymbol{s})$ is not summable, it is
not difficult to verify that $X$ is a dense $G_{\delta }$ subset of $\left(
0,1\right) ^{\mathbb{N}}$. The restriction $f|_{X}$ of $f$ to $X$ is a Borel
reduction from $F|_{X}$ to $E_{\mathcal{A}_{0}}$. It is thus enough to show
that $F|_{X}$ is not classifiable by countable structures.

\begin{lemma}
$F$ has meager classes
\end{lemma}

\begin{proof}
Fix $\boldsymbol{s}\in (0,1)$. We want to show that the $F$-class of $%
\boldsymbol{s}$ is meager. We can assume without loss of generality that $%
\prod_{i\leq k}s_{i}\rightarrow 0$ for $k\rightarrow \infty $, as the set of
such $\boldsymbol{s}$ is a comeager subset of $(0,1)^{\mathbb{N}}$. Fix $%
m\in \mathbb{N}$ and let $K_{m}$ be the (closed) set of $\boldsymbol{t}\in
(0,1)^{\mathbb{N}}$ such that, for every $n\in \mathbb{N}$,%
\begin{equation*}
\left\vert \sum_{k<n}\left( \prod_{i\leq k}s_{i}-\prod_{i\leq k}t_{i}\right)
\right\vert \leq m\text{.}
\end{equation*}%
Observe that if $\boldsymbol{t}^{0}\in K$ and $n_{0}\in \mathbb{N}$ then the
element $\boldsymbol{t}$ of $(0,1]^{\mathbb{N}}$ defined by 
\begin{equation*}
t_{i}=\left\{ 
\begin{array}{ll}
t_{i}^{0} & \text{for }i\leq n_{0}\text{,} \\ 
1-2^{-i} & \text{otherwise}%
\end{array}%
\right.
\end{equation*}%
does not belong to $K_{m}$. Therefore $K_{m}$ is nowhere dense. Finally
observe that the $F$-class of $\boldsymbol{s}$ is $\bigcup_{m}K_{m}$.
\end{proof}

Let now $\Gamma $ be the subgroup of $\mathbb{R}_{+}^{\mathbb{N}}$
containing those sequences $\boldsymbol{g}$ such that%
\begin{equation*}
\sum_{n}\left\vert \prod_{k\leq n}g_{k}-1\right\vert <\infty \text{.}
\end{equation*}%
Observe that $\Gamma $ is indeed a subgroup of $\mathbb{R}_{+}^{\mathbb{N}}$%
. In fact suppose that $\boldsymbol{g},\boldsymbol{h}\in \Gamma $. Fix $%
n_{0}\in \mathbb{N}$ such that 
\begin{equation*}
\left\vert \prod_{k\leq n}g_{k}-1\right\vert \leq \frac{1}{2}
\end{equation*}%
for every $n\geq n_{0}$. Then%
\begin{equation*}
\sum_{n\in \mathbb{N}}\left\vert \prod_{k\leq n}g_{k}^{-1}-1\right\vert \leq
\sum_{n<n_{0}}\left\vert \prod_{k\leq n}g_{k}^{-1}-1\right\vert +2\sum_{n\in 
\mathbb{N}}\left\vert \prod_{k\leq n}g_{k}-1\right\vert <\infty
\end{equation*}%
and hence $\boldsymbol{g}^{-1}\in \Gamma $. Furthermore%
\begin{equation*}
\sum_{n\in \mathbb{N}}\left\vert \prod_{k\leq n}g_{k}h_{k}-1\right\vert \leq
\sum_{n<n_{0}}\left\vert \prod_{k\leq n}g_{k}h_{k}-1\right\vert +\frac{3}{2}%
\sum_{n\in \mathbb{N}}\left\vert \prod_{k\leq n}h_{k}-1\right\vert
+\sum_{n}\left\vert \prod_{k\leq n}g_{k}-1\right\vert <\infty
\end{equation*}%
and hence $\boldsymbol{gh}\in \Gamma $. Consider the coordinate-wise
multiplication action $\Gamma \curvearrowright \mathbb{R}_{+}^{\mathbb{N}}$.
Let $E$ be the restriction to $\left( 0,1\right) ^{\mathbb{N}}$ of the orbit
equivalence relation of the action $\Gamma \curvearrowright \mathbb{R}_{+}^{%
\mathbb{N}}$. Since $F$ has meager classes, $X\subset (0,1)^{\mathbb{N}}$ is
comeager, and $E\subset F$, in order to prove that $F|_{X}$ is not
classifiable by countable structures it is enough to show that $E$ is
generically $S_{\infty }$-ergodic.

Define the bi-invariant metric $d_{\Gamma }$ on $\Gamma $ by setting%
\begin{equation*}
d_{\Gamma }\left( \boldsymbol{g},\boldsymbol{h}\right) =\sum_{n\in \mathbb{N}%
}\left\vert \prod_{k\leq n}g_{k}-\prod_{k\leq n}h_{k}\right\vert \text{.}
\end{equation*}%
We claim that $d_{\Gamma }$ induces a Polish topology on $\Gamma $. To this
end, consider the injective map $\Phi :\mathbb{R}_{+}^{\mathbb{N}%
}\rightarrow \mathbb{R}^{\mathbb{N}}$ defined by%
\begin{equation*}
\boldsymbol{a}\mapsto \left( \left( \prod_{k\leq n}a_{k}\right) -1\right)
_{n}\text{.}
\end{equation*}%
Observe that the restriction of $\Phi $ to $\Gamma $ is an isometry from $%
\left( \Gamma ,d_{\Gamma }\right) $ to $\ell ^{1}$ endowed with the $\ell
^{1}$-metric. Furthermore the image of $\Gamma $ under $\Phi $ is a $%
G_{\delta }$ subset of $\ell ^{1}$, since $\boldsymbol{b}\in \Phi \left[
\Gamma \right] $ if and only if $b_{n}>-1$ for every $n\in \mathbb{N}$.
Since a $G_{\delta }$ subspace of a Polish space is Polish \cite[Theorem 3.11%
]{kechris_classical_1995}, this concludes the proof that $d_{\Gamma }$
induces a Polish topology on $\Gamma $.

If $\boldsymbol{g}\in \Gamma$ and $\boldsymbol{s}\in (0,1)^{\mathbb{N}}$,
define $\boldsymbol{gs}\in \mathbb{R}_{+}^{\mathbb{N}}$ by%
\begin{equation*}
\left( \boldsymbol{gs}\right) _{n}=g_{n}s_{n}\text{.}
\end{equation*}%
Consider now the groupoid%
\begin{equation*}
G=\left\{ \left( \boldsymbol{g},\boldsymbol{s}\right) \in \Gamma \times
(0,1)^{\mathbb{N}}:\boldsymbol{gs}\in \left( 0,1\right) ^{\mathbb{N}%
}\right\} \text{.}
\end{equation*}%
Composition and inversion of arrows in $G$ are defined by%
\begin{equation*}
\left( \boldsymbol{g},\boldsymbol{s}\right) \left( \boldsymbol{h},%
\boldsymbol{t}\right) =\left( \boldsymbol{gh},\boldsymbol{t}\right)
\end{equation*}%
whenever $\boldsymbol{ht}=\boldsymbol{s}$, and%
\begin{equation*}
\left( \boldsymbol{g},\boldsymbol{s}\right) ^{-1}=\left( \boldsymbol{g}^{-1},%
\boldsymbol{gs}\right) \text{.}
\end{equation*}%
Being a closed subset of $\Gamma \times (0,1)^{\mathbb{N}}$, $G$ is Polish
with the induced topology. Clearly composition and inversion of arrows are
continuous.\ Furthermore the map $\left( \boldsymbol{1},\boldsymbol{s}%
\right) \mapsto \boldsymbol{s}$ allows one to identify the set of objects of 
$G$ with $\left( 0,1\right) ^{\mathbb{N}}$. It remains to show that
composition of arrows is open. To this purpose it is enough to show that the
source map%
\begin{eqnarray*}
G &\rightarrow &(0,1)^{\mathbb{N}} \\
\left( \boldsymbol{g},\boldsymbol{s}\right) &\mapsto &\boldsymbol{s}
\end{eqnarray*}%
is open; see \cite[Exercise I.1.8]{resende_lectures_2006}. Suppose that $%
\left( \boldsymbol{g},\boldsymbol{s}\right) \in G$, and $U$ is an open
neighborhood of $\left( \boldsymbol{g},\boldsymbol{s}\right) $. Thus there
exist $\varepsilon >0$ and $N\in \mathbb{N}$ such that $U$ contains all the
pairs $\left( \boldsymbol{h},\boldsymbol{t}\right) \in G $ such that $%
d_{\Gamma }\left( \boldsymbol{g},\boldsymbol{h}\right) <\varepsilon $ and $%
\left\vert s_{n}-t_{n}\right\vert <\varepsilon $ for $n\leq N$. Suppose that 
$\varepsilon > \eta >0$ is such that $g_{n}\left( s_{n}+\eta \right) <1$ for
every $n\leq N$. Consider the neighborhood $W$ of $\boldsymbol{s}$
consisting of those $\boldsymbol{t}\in \left( 0,1\right) ^{\mathbb{N}}$ such
that $\left\vert s_{n}-t_{n}\right\vert <\eta $ for every $n \le N$. We
claim that $s\left[ U\right] \supset W$. In fact if $\boldsymbol{t}\in W$ we
have that for $n\leq N$,%
\begin{equation*}
g_{n}t_{n}\leq g_{n}\left( s_{n}+\eta \right) <1
\end{equation*}%
and therefore $\left( \boldsymbol{g},\boldsymbol{t}\right) \in U$.

In the following lemma we establish that $G$ is a turbulent Polish groupoid.
Together with Theorem \ref{Theorem:turbulence-groupoids}, this implies that
its associated orbit equivalence relation $E$ is generically $S_{\infty }$%
-ergodic, concluding the proof of Proposition \ref%
{Proposition:nonclassify-sequences}.

\begin{lemma}
Any element $\boldsymbol{s}$ of $\left( 0,1\right) ^{\mathbb{N}}$ is a
turbulent object with dense orbit for the Polish groupoid $G$.
\end{lemma}

\begin{proof}
It is easy to see that the orbit of $\boldsymbol{s}$ is dense. It remains to
show that for any neighborhood $V$ of $\left( \boldsymbol{1},\boldsymbol{s}%
\right) $ in $G$ the local orbit $\mathcal{O}\left( \boldsymbol{s},V\right) $
is somewhere dense. Without loss of generality we can assume that there
exist $\varepsilon >0$ and $n_{0}\in \mathbb{N}$ such that, if 
\begin{equation*}
U=\left\{ \boldsymbol{t}\in \left( 0,1\right) ^{\mathbb{N}}:\forall n\leq
n_{0},{}\left\vert \frac{t_{n}}{s_{n}}-1\right\vert <\varepsilon \right\}
\end{equation*}%
and%
\begin{equation*}
W=\left\{ \boldsymbol{g}\in \Gamma :d_{\Gamma }\left( \boldsymbol{g},%
\boldsymbol{1}\right) <\varepsilon \right\} \text{,}
\end{equation*}%
then $V=\left( W\times U\right) \cap G$. We claim that the local orbit $%
\mathcal{O}\left( \boldsymbol{s},V\right) $ dense in $U$. Fix $\boldsymbol{t}%
\in U$ and $n_{1}\geq n_{0}$. Let $N\in \mathbb{N}$, to be determined later.
Set 
\begin{equation*}
g_{k}=\left\{ 
\begin{array}{ll}
\sqrt[N]{t_{k}/s_{k}} & \text{for }k\leq n_{1}\text{,} \\ 
\prod_{j\leq n_{1}}\sqrt[N]{s_{j}/t_{j}} & \text{for }k=n_{1}+1\text{,} \\ 
1 & \text{otherwise.}%
\end{array}%
\right.
\end{equation*}%
Observe that, for $N$ large enough, we have that $\boldsymbol{g}\in \Gamma $%
, $d\left( \boldsymbol{g},\boldsymbol{1}\right) <\varepsilon $, and $%
\boldsymbol{g}^{i}\boldsymbol{s}\in U$ for every $i\leq N$. Finally observe
that $g_{k}^{N}s_{k}=t_{k}$ for $k\leq n_{1}$. This concludes the proof that
the local orbit $\mathcal{O}\left( \boldsymbol{s},V\right) $ is dense in $U$%
. Since this is true for every neighborhood $V$ of $\boldsymbol{s}$ in $G$, $%
\boldsymbol{s}$ is a turbulent point for $G$.
\end{proof}

\section{Conformal equivalence of operator algebraic varieties\label%
{Section:conformal}}

Fix $d\in \mathbb{N}$ and let $\mathcal{V}_{d}$ be the space of operator
algebraic varieties in $\mathbb{B}_{d}$. Denote by $\Aut(\mathbb{B}_{d})$ be
the group of conformal automorphisms of $\mathbb{B}_{d}$. The \emph{%
pseudo-hyperbolic distance} $\rho $ on $\mathbb{B}_{d}$ is defined by 
\begin{equation*}
\rho \left( a,b\right) =\left\Vert \varphi _{a}\left( b\right) \right\Vert ,
\end{equation*}%
where $\left\Vert \cdot \right\Vert $ is the usual Euclidean norm and $%
\varphi _{a}$ is the conformal automorphism of $\mathbb{B}_{d}$ which
interchanges $0$ and $a$ defined in \cite[Subsection 2.2.1]%
{rudin_function_2008}. Recall that $\rho $ is a proper metric (since its
closed balls coincide with Euclidean closed balls) that induces the usual
topology in $\mathbb{B}_{d}$. Furthermore, $\Aut(\mathbb{B}_{d})$ is a
closed subgroup of the group of isometries of $\left( \mathbb{B}_{d},\rho
\right) $, and hence a locally compact Polish group when endowed with the
compact-open topology. More information about conformal automorphisms of $%
\mathbb{B}_{d}$ can be found in \cite[Chapter 2]{rudin_function_2008}.
Consider the Borel action of $\Aut(\mathbb{B}_{d})$ on $\mathcal{V}_{d}$
defined by $\left( \alpha ,V\right) \mapsto \alpha \left[ V\right] $.
Observe that the relation $E_{\Aut(\mathbb{B}_{d})}^{\mathcal{V}_{d}}$ of $%
\Aut(\mathbb{B}_{d})$-conformal equivalence of operator algebraic varieties
in $\mathbb{B}_{d}$ is the orbit equivalence relation associated with this
action. Therefore, it follows from \cite[Theorem 1.1]{kechris_countable_1992}
that $E_{\Aut(\mathbb{B}_{d})}^{\mathcal{V}_{d}}$ is essentially countable.

The rest of this section is devoted to prove Theorem \ref{Theorem:countable}%
, asserting that $E_{\Aut(\mathbb{B}_{d})}^{\mathcal{V}_{d}}$ has in fact
maximum complexity among essentially countable equivalence relations. As
explained in the introduction, the same conclusion will then apply to the
relation of (completely) isometric isomorphisms of multiplier algebras $%
\mathrm{Mult}(H_{\infty }^{2}|_{V})$ for $V\in \mathcal{V}_{d}$.

Observe that the canonical inclusion of $\mathbb{B}_{d}$ into $\mathbb{B}%
_{d+1}$ induces an inclusion of $\mathcal{V}_{d}$ into $\mathcal{V}_{d+1}$.
According to the following proposition, this inclusion is a Borel reduction
from the relation of $\Aut(\mathbb{B}_d)$-conformal equivalence on $\mathcal{%
V}_{d}$ to the relation of $\Aut(\mathbb{B}_{d+1})$-conformal equivalence on 
$\mathcal{V}_{d+1}$. We mention that this result also follows from \cite[%
Theorem 4.4]{davidson_operator_2015}.

\begin{proposition}
Let $X,Y \subset \mathbb{B}_d$ be subsets. Then $X$ and $Y$ are conformally
equivalent via an element of $\Aut(\mathbb{B}_d)$ if and only if they are
conformally equivalent via an element of $\Aut(\mathbb{B}_{d+1})$.
\end{proposition}

\begin{proof}
By \cite[Section 2.2.8]{rudin_function_2008}, every conformal automorphism
of $\mathbb{B}_d$ extends to a conformal automorphism of $\mathbb{B}_{d+1}$.
This establishes one direction.

Conversely, suppose that $F \in \Aut(\mathbb{B}_{d+1})$ maps $X$ onto $Y$,
and let $G \subset \Aut(\mathbb{B}_{d+1})$ denote the subgroup of all
automorphisms which fix $\mathbb{B}_d$. We wish to show that $X$ and $Y$ are 
$G$-equivalent. Since $\Aut(\mathbb{B}_d)$ acts transitively on $\mathbb{B}%
_d $ \cite[Theorem 2.2.3]{rudin_function_2008}, and since every element of $%
\Aut(\mathbb{B}_d)$ extends to an element of $G$, the subgroup $G$ acts
transitively on $\mathbb{B}_d$. We may therefore assume that $0 \in X$ and $%
0 \in Y$. By Proposition 2.4.2 in \cite{rudin_function_2008} and the
discussion preceding it, $F$ maps the affine span of $X$ onto the affine
span of $Y$. Hence, $F$ maps $\mathrm{span}(X) \cap \mathbb{B}_d$ onto $%
\mathrm{span}(Y) \cap \mathbb{B}_d$, where $\mathrm{span}$ denotes the
linear span. Since $\mathrm{span}(X) \cap \mathbb{B}_d$ and $\mathrm{span}%
(Y) \cap \mathbb{B}_d$ are themselves unitarily equivalent to complex balls
of dimension $e \le d$, and since automorphisms of $\mathbb{B}_e$ extend to
automorphisms of higher-dimensional balls, we conclude that there exists a
map $\widetilde F \in G$ such that $F |_{\mathrm{span}(X) \cap \mathbb{B}_d}
= \widetilde F |_{\mathrm{span}(X) \cap \mathbb{B}_d}$. This completes the
proof.
\end{proof}

Therefore to establish the desired lower bound on the complexity of $E_{\Aut(%
\mathbb{B}_{d})}^{\mathcal{V}_{d}}$ it suffices to consider the case $d=1$.
The $1$-dimensional complex Euclidean ball is usually called the \emph{disc }%
and denoted by $\mathbb{D}$. The elements of $\mathcal{V}_{1}\setminus \{%
\mathbb{D}\}$ are called \emph{Blaschke sequences}. For background material
on Blaschke sequences, the reader is referred to \cite[Chapter II, Section 2]%
{garnett_bounded_2007}. Recall from Section \ref{Section:complexity} that
the orbit equivalence relation $E(F_{2},2)$ associated with the left
translation action of the free group $F_{2}$ on its subsets has maximum
complexity among essentially countable equivalence relation. We will now
show that $E(F_{2},2)$ is Borel reducible to the relation $E_{\Aut(\mathbb{B}%
_{d})}^{\mathcal{V}_{d}}$ of $\Aut(\mathbb{D})$-conformal equivalence of
Blaschke sequences. To this end, we will adapt the proof of \cite[Theorem 4.1%
]{hjorth_complexity_2000}.

The lower bound in \cite[Theorem 4.1]{hjorth_complexity_2000} is achieved by
encoding the action of $F_{2}$ on $\left\{ 0,1\right\} ^{F_{2}}$ by
translation. The crucial point in this proof is that $\mathrm{Aut}(\mathbb{D}%
)$ contains a copy of $F_{2}$ such that the orbit of every point in $\mathbb{%
D}$ is discrete. We require something stronger, namely that the orbit of
every point is a Blaschke sequence.

\begin{proposition}
\label{prop:schottky} There exists a discrete group $\Gamma \subset \mathrm{%
Aut}(\mathbb{D})$ which is isomorphic to $F_{2}$ such that 
\begin{equation*}
\sum_{g\in \Gamma }\left( 1-|g(z)|\right) <\infty
\end{equation*}%
for every $z\in \mathbb{D}$.
\end{proposition}

\begin{proof}
Let $g_{1}$ and $g_{2}$ be two conformal automorphisms of $\mathbb{D}$ which
generate a Schottky group (see Chapter II, Section 1 in \cite%
{dalbo_geodesic_2011}), and let $\Gamma $ be the group generated by $g_{1}$
and $g_{2}$. Then $\Gamma $ is isomorphic to $F_{2}$ by \cite[Chapter II,
Proposition 1.6]{dalbo_geodesic_2011}. By the same proposition, the closure
of the Dirichlet domain $\mathcal{D}_{0}(\Gamma )$ of $\Gamma $ contains
nontrivial arcs in $\partial \mathbb{D}$ (see \cite[Chapter I, Section 2.3]%
{dalbo_geodesic_2011} for the definition of the Dirichlet domain). In
particular, the Lebesgue measure of $\overline{\mathcal{D}_{0}(\Gamma )}\cap
\partial \mathbb{D}$ is strictly positive. In this situation, \cite[Theorem
XI.4]{tsuji_potential_1975} applies to show that 
\begin{equation*}
\sum_{g\in \Gamma }\left( 1-|g(0)|\right) <\infty .
\end{equation*}%
Finally, the argument preceding Theorem XI.3 in \cite{tsuji_potential_1975}
shows that this sum is finite if $0$ is replaced with an arbitrary point $%
z\in \mathbb{D}$.
\end{proof}

It seems worthwhile to give a concrete example of two conformal
automorphisms of $\mathbb{D}$ which generate a group $\Gamma $ as in the
statement of the proposition. Let $\mathbb{H}$ denote the upper half-plane
in $\mathbb{C}$. Recall that $\mathbb{D}$ and $\mathbb{H}$ are conformally
equivalent via the Cayley map 
\begin{eqnarray*}
\mathbb{H} &\rightarrow &\mathbb{D} \\
z &\mapsto &\frac{z-i}{z+i}\text{.}
\end{eqnarray*}%
This map induces an isomorphism of topological groups between $\mathrm{Aut}(%
\mathbb{D})$ and $\mathrm{Aut}(\mathbb{H})$. Moreover, $\mathrm{Aut}(\mathbb{%
H})$ is isomorphic to $\mathrm{PSL}_{2}(\mathbb{R})$ via the map that
assigns to the matrix $%
\begin{pmatrix}
a & b \\ 
c & d%
\end{pmatrix}%
\in \mathrm{PSL}_{2}(\mathbb{R})$ the corresponding \emph{M\"{o}bius
transformation}%
\begin{equation*}
z\mapsto \frac{az+b}{cz+d}{}\text{.}
\end{equation*}%
Let $\Phi :\mathrm{PSL}_{2}(\mathbb{R})\rightarrow \mathrm{Aut}(\mathbb{D})$
denote the isomorphism obtained by composing the two isomorphisms above. The
group $\Lambda $ considered in the proof of \cite[Theorem 4.1]%
{hjorth_complexity_2000} is generated by the images of 
\begin{equation*}
\begin{pmatrix}
1 & 2 \\ 
0 & 1%
\end{pmatrix}%
\quad \text{and}\quad 
\begin{pmatrix}
1 & 0 \\ 
2 & 1%
\end{pmatrix}%
\end{equation*}%
under $\Phi $. The group $\Lambda $ is isomorphic to $F_{2}$, but the orbit
of $0$ under $\Lambda $ is not a Blaschke sequence. This follows from the
following facts:

\begin{itemize}
\item $\Lambda $ has finite index in $\mathrm{PSL}(2,\mathbb{Z})$, and

\item the orbit of $0$ under $\Phi \left[ \mathrm{PSL}_{2}(\mathbb{Z})\right]
$ is not a Blaschke sequence, as its conical limit set on $\partial \mathbb{D%
}$ has positive Lebesgue measure; see \cite[Chapter II, Section 3.1]%
{dalbo_geodesic_2011}.
\end{itemize}

Moreover, $\Lambda $ is not a Schottky group, but just a generalized
Schottky group in the sense of \cite[Chapter II, Section 1.1]%
{dalbo_geodesic_2011}. However, if we let $\Gamma \subset \mathrm{Aut}(%
\mathbb{D})$ denote the group generated by the images of 
\begin{equation*}
\begin{pmatrix}
1 & 3 \\ 
0 & 1%
\end{pmatrix}%
\quad \text{and}\quad 
\begin{pmatrix}
1 & 0 \\ 
3 & 1%
\end{pmatrix}%
\text{,}
\end{equation*}%
then it is not hard to see that $\Gamma $ is indeed a Schottky group, and
thus satisfies the conclusion of the proposition.

In the proof of the next theorem, we require the following elementary
observation.

\begin{lemma}
\label{lem:triangle} Let $(X,d)$ be a metric space and let $%
x^{(0)},x^{(1)},x^{(2)}$ and $y^{(0)},y^{(1)},y^{(2)},y^{(3)}$ by points in $%
X$ such that 
\begin{equation*}
d(x^{(i)},x^{(j)})=d(y^{(i)},y^{(j)})
\end{equation*}%
for $0\leq i,j\leq 2$ and such that the distances $d(y^{(i)},y^{(j)})$ are
all distinct for $0\leq i<j\leq 3$. If $\theta :X\rightarrow X$ is an
isometry such that 
\begin{equation*}
\theta (\{x^{(0)},x^{(1)},x^{(2)}\})\subset
\{y^{(0)},y^{(1)},y^{(2)},y^{(3)}\},
\end{equation*}%
then $\theta (x^{(i)})=y^{(i)}$ for $0\leq i\leq 2$.
\end{lemma}

\begin{proof}
The assumptions on the distances and the fact that $\theta$ is an isometry
imply that 
\begin{align*}
\theta( \{ x^{(0)}, x^{(1)} \}) &= \{y^{(0)}, y^{(1)} \}, \\
\theta( \{ x^{(0)}, x^{(2)} \}) &= \{y^{(0)}, y^{(2)} \}, \text{ and } \\
\theta( \{ x^{(1)}, x^{(2)} \}) &= \{y^{(1)}, y^{(2)} \}.
\end{align*}
This is only possible if $\theta(x^{(i)}) = y^{(i)}$ for $0 \le i \le 2$.
\end{proof}

We are now ready to prove the main result of this section.

\begin{theorem}
The relation $E(F_{2},2)$ is Borel reducible to the relation of $\Aut(%
\mathbb{D})$-conformal equivalence of Blaschke sequences.
\end{theorem}

\begin{proof}
The proof is an adaptation of the proof of the lower bound in \cite[Theorem
4.1]{hjorth_complexity_2000}. The details are as follows.

Let $\Gamma$ be a group as in Proposition \ref{prop:schottky}. We will
identify $F_2$ with $\Gamma$. Moreover, let $\rho$ be the pseudohyperbolic
metric on $\mathbb{D}$, and for $z \in \mathbb{D}$ and $\varepsilon > 0$,
let 
\begin{equation*}
D_\varepsilon(z) = \{ y \in \mathbb{D}: \rho(y,z) < \varepsilon\}.
\end{equation*}
We will explicitly construct four Blaschke sequences 
\begin{equation*}
B_i = \{ x_g^{(i)} : g \in F_2 \}
\end{equation*}
for $0 \le i \le 3$ and find $\varepsilon > 0$ with the following properties:

\begin{enumerate}
\item $gx_{h}^{(i)}=x_{gh}^{(i)}$ for $g,h\in F_{2}$ and $0\leq i\leq 3$;

\item $x_{g}^{(i)}\in D_{\varepsilon /5}(x_{g}^{(0)})$ for $g\in F_{2}$ and $%
0\leq i\leq 3$;

\item $D_{\varepsilon /2}(x_{g}^{(0)})\cap (B_{0}\cup B_{1}\cup B_{2}\cup
B_{3})=\{x_{g}^{(i)}:0\leq i\leq 3\}$;

\item The distances $\rho (x_{g}^{(i)},x_{g}^{(j)})$ do not depend on $g\in
F_{2}$ and are all distinct and positive for $0\leq i<j\leq 3$.
\end{enumerate}

The construction proceeds as follows. Let $x_{1}^{(0)}\in \mathbb{D}$ be
arbitrary and set $x_{g}^{(0)}=g(x_{1}^{(0)})$ for $g\in F_{2}$. Let $%
B_{0}=\{x_{g}^{(0)}:g\in F_{2}\}$. Then $B_{0}$ is a Blaschke sequence. In
particular, there exists $\varepsilon >0$ such that 
\begin{equation*}
D_{\varepsilon }(x_{1}^{(0)})\cap B_{0}=\{x_{1}^{(0)}\}.
\end{equation*}%
Choose distinct points $x_{1}^{(i)}\in D_{\varepsilon
/5}(x_{1}^{(0)})\setminus \{x_{1}^{(0)}\}$ for $i\in \{1,2,3\}$ such that
the pseudohyperbolic distances $\rho (x_{1}^{(i)},x_{1}^{(j)})$ for $i<j$
are all different from each other, and define $x_{g}^{(i)}=g(x_{1}^{(i)})$
for $i\in \{1,2,3\}$ and $g\in F_{2}$. Moreover, set $B_{i}=\{x_{g}^{(i)}:g%
\in F_{2}\}$. Using the fact that every $g\in F_{2}$ is an isometry with
respect to $\rho $, properties (1)--(4) are now easy to verify.

Given $A\subset F_{2}$, let 
\begin{equation*}
V_{A}=B_{0}\cup B_{1}\cup B_{2}\cup \{x_{g}^{(3)}:g\in A\}.
\end{equation*}%
We will show that $A=gB$ for some $g\in F_{2}$ if and only if $V_{A}$ and $%
V_{B}$ are $\Aut(\mathbb{D})$-conformally equivalent. Clearly, if $g\in
F_{2} $ such that $gA=B$, then $g[V_{A}]=V_{B}$, hence $V_{A}$ and $V_{B}$
are $\Aut(\mathbb{D})$-conformally equivalent. Conversely, assume that there
exists $\theta \in \mathrm{Aut}(\mathbb{D})$ with $\theta \lbrack
V_{A}]=V_{B}$. We will show that there exists $g\in F_{2}$ such that $\theta
=g$. Since $x_{1}^{(0)}\in V_{A}$, there exists $g\in F_{2}$ and $i\in
\{0,1,2,3\}$ such that $\theta (x_{1}^{(0)})=x_{g}^{(i)}$. Observe that for $%
k\in \{1,2\}$, we have 
\begin{equation*}
\rho (\theta (x_{1}^{(k)}),x_{g}^{(i)})=\rho (\theta (x_{1}^{(k)}),\theta
(x_{1}^{(0)}))=\rho (x_{1}^{(k)},x_{1}^{(0)})<\varepsilon /5
\end{equation*}%
by Condition (2). By the same condition, $\rho
(x_{g}^{(i)},x_{g}^{(0)})<\varepsilon /5$, hence 
\begin{equation*}
\theta (x_{1}^{(k)})\in D_{\varepsilon /2}(x_{g}^{(0)}).
\end{equation*}%
Therefore, Condition (3) implies that 
\begin{equation*}
\theta (\{x_{1}^{(0)},x_{1}^{(1)},x_{1}^{(2)}\})\subset \{x_{g}^{(i)}:0\leq
i\leq 3\}.
\end{equation*}%
In light of Condition (4), an application of Lemma \ref{lem:triangle} shows
that $\theta (x_{1}^{(i)})=x_{g}^{(i)}$ for $0\leq i\leq 2$. This means that 
$\theta $ and $g$ are two M\"{o}bius transformations which agree on three
points. Consequently, $\theta =g$; see for example \cite[Theorem 10.10]%
{priestley_introduction_2003}. We finish the proof by showing that $gA=B$.
Note that if $h\in A$, then $x_{h}^{(3)}\in V_{A}$. Therefore, $%
x_{gh}^{(3)}=g(x_{h}^{(3)})=\theta (x_{h}^{(3)})\in V_{B}$, so $gh\in B$.
This shows that $gA\subset B$. Similarly, $g^{-1}B\subset A$, so $gA=B$, as
desired.
\end{proof}

\section*{Appendix}

Recall that if $X$ is a Polish space, then the space $F(X)$ of nonempty
closed subsets of $X$ is a standard Borel space when endowed with the Effros
Borel structure \cite[Section 12.C]{kechris_classical_1995}. This is the
Borel structure generated by the sets%
\begin{equation*}
\left\{ K\in F(X):K\cap U\neq \varnothing \right\}
\end{equation*}%
where $U$ ranges over the open subsets of $X$. The
Kuratowski--Ryll-Nardzewski theorem asserts that there exists a sequence $%
\left( \sigma _{n}\right) $ of Borel maps from $F(X)$ to $X$ such that $%
\left( \sigma _{n}(A)\right) $ enumerates a dense subset of $A$ for every
nonempty closed subset $A$ of $X$ \cite[Theorem 12.13]%
{kechris_classical_1995}.

Fix $d\in \mathbb{N}\cup \left\{ \infty \right\} $, and let $\mathcal{V}%
_{d}\subset F(\mathbb{B}_{d})$ be the set of operator algebraic varieties in 
$\mathbb{B}_{d}$. For $d\leq d^{\prime }$ the canonical inclusion $\mathbb{B}%
_{d}\subset \mathbb{B}_{d^{\prime }}$ induces a Borel injection from $%
\mathcal{V}_{d}$ into $\mathcal{V}_{d^{\prime }}$.

\begin{proposition*}
The set $\mathcal{V}_{d}$ of operator algebraic varieties in $\mathbb{B}_{d}$
is a Borel subset of $F(\mathbb{B}_{d})$.
\end{proposition*}

\begin{proof}
Let $H_{d}^{2}$ be the Drury-Arveson space. Observe that $\mathcal{V}_{d}$
is the image of $F(H_{d}^{2})$ under the Borel map that assigns to a closed
subset $S$ of $H_{d}^{2}$ the operator algebraic variety $V_{S}$ of common
zeros of elements of $S$. Therefore $\mathcal{V}_{d}$ is analytic. By \cite[%
Theorem 14.7]{kechris_classical_1995} it remains to show that $\mathcal{V}%
_{d}$ is coanalytic. Fix a dense subset $\mathcal{D}$ in the unit ball of $%
H_{d}^{2}$, and a sequence $\left( \sigma _{n}\right) $ of Borel functions
from $F(\mathbb{B}_{d})$ to $\mathbb{B}_{d}$ such that $\left( \sigma
_{n}(S)\right) $ enumerates a dense subset of $S$ for every nonempty closed
subset $S$ of $\mathbb{B}_{d}$. We claim that a closed subset $V$ of $%
\mathbb{B}_{d}$ is an operator algebraic variety if and only if for every $%
x\in \mathbb{B}_{d}$ either $x\in V$ or there exists a strictly positive
rational number $\varepsilon $ such that for every $n\in \mathbb{N}$ there
exists $f\in \mathcal{D}$ such that $\left\vert f(x)\right\vert \geq
\varepsilon $ and $\left\vert f\left( \sigma _{i}(V)\right) \right\vert \leq
2^{-n}$ for every $i\leq n$.

One implication is obvious. For the other implication it is enough to
observe that if $x\in \mathbb{B}_{d}$ and $\left( f_{n}\right) $ is a
sequence in the unit ball of $H_{d}^{2}$, then $\left( f_{n}\right) $ has an
accumulation point $f$ in the unit ball of $H_{d}^{2}$ with respect to the
weak topology. If furthermore $\left\vert f_{n}(x)\right\vert \geq
\varepsilon $ and $\left\vert f\left( \sigma _{i}(V)\right) \right\vert \leq
2^{-n}$ for every $i\leq n$ then $f(x)\neq 0$ and $f$ vanishes on $V$. This
shows that $V$ is the set of common zeroes of all the elements of $H_{d}^{2}$
that vanish on $V$, and hence it is an operator algebraic variety.
\end{proof}

\bibliographystyle{amsplain}
\bibliography{NP-bib}

\providecommand{\bysame}{\leavevmode\hbox to3em{\hrulefill}\thinspace}
\providecommand{\MR}{\relax\ifhmode\unskip\space\fi MR }
\providecommand{\MRhref}[2]{%
  \href{http://www.ams.org/mathscinet-getitem?mr=#1}{#2}
}
\providecommand{\href}[2]{#2}
\begin{thebibliography}{10}

\bibitem{agler_complete_2000}
Jim Agler and John~E. MCarthy, \emph{Complete {Nevanlinna}-{Pick} {Kernels}},
  Journal of Functional Analysis \textbf{175} (2000), no.~1, 111--124.

\bibitem{agler_pick_2002}
Jim Agler and John~E. McCarthy, \emph{Pick interpolation and {Hilbert} function
  spaces}, Graduate {Studies} in {Mathematics}, vol.~44, American Mathematical
  Society, Providence, RI, 2002.

\bibitem{alpay_hilbert_2003}
Daniel Alpay, Mihai Putinar, and Victor Vinnikov, \emph{A {H}ilbert space
  approach to bounded analytic extension in the ball}, Communications on Pure
  and Applied Analysis \textbf{2} (2003), no.~2, 139--145.

\bibitem{arcozzi_carleson_2008}
Nicola Arcozzi, Richard Rochberg, and Eric~T. Sawyer, \emph{Carleson measures
  for the {Drury}--{Arveson} {Hardy} space and other {Besov}--{Sobolev} spaces
  on complex balls}, Advances in Mathematics \textbf{218} (2008), no.~4,
  1107--1180.

\bibitem{arveson_subalgebras_1998}
William Arveson, \emph{Subalgebras of {C}*-algebras {III}: Multivariable
  operator theory}, Acta Mathematica \textbf{181} (1998), no.~2, 159--228.

\bibitem{blackadar_k-theory_1998}
Bruce Blackadar, \emph{K-theory for operator algebras}, second ed.,
  Mathematical Sciences Research Institute Publications, vol.~5, Cambridge
  University Press, Cambridge, 1998.

\bibitem{dalbo_geodesic_2011}
Fran{\c{c}}oise Dal'Bo, \emph{Geodesic and horocyclic trajectories},
  Universitext, Springer-Verlag London, Ltd., London; EDP Sciences, Les Ulis,
  2011.

\bibitem{davidson_operator_2015}
Kenneth Davidson, Christopher Ramsey, and Orr Shalit, \emph{Operator algebras
  for analytic varieties}, Transactions of the American Mathematical Society
  \textbf{367} (2015), no.~2, 1121--1150.

\bibitem{davidson_multipliers_2014}
Kenneth~R. Davidson, Michael Hartz, and Orr~Moshe Shalit, \emph{Multipliers of
  embedded discs}, Complex Analysis and Operator Theory \textbf{9} (2014),
  no.~2, 287--321, Erratum, ibid., p 323--327.

\bibitem{davidson_isomorphism_2011}
Kenneth~R. Davidson, Christopher Ramsey, and Orr~Moshe Shalit, \emph{The
  isomorphism problem for some universal operator algebras}, Advances in
  Mathematics \textbf{228} (2011), no.~1, 167--218.

\bibitem{friedman_borel_1989}
Harvey Friedman and Lee Stanley, \emph{A {Borel} reducibility theory for
  classes of countable structures}, Journal of Symbolic Logic \textbf{54}
  (1989), no.~3, 894--914.

\bibitem{gao_invariant_2009}
Su~Gao, \emph{Invariant descriptive set theory}, Pure and Applied Mathematics
  (Boca Raton), vol. 293, {CRC} Press, Boca Raton, {FL}, 2009.

\bibitem{garnett_bounded_2007}
John~B. Garnett, \emph{Bounded analytic functions}, first ed., Graduate Texts
  in Mathematics, vol. 236, Springer, New York, 2007. \MR{2261424}

\bibitem{hartz_isomorphism_2015}
Michael Hartz, \emph{On the isomorphism problem for multiplier algebras of
  {Nevanlinna}-{Pick} spaces}, Canadian Journal of Mathematics, to appear.

\bibitem{hartz_topological_2012}
\bysame, \emph{Topological isomorphisms for some universal operator algebras},
  Journal of Functional Analysis \textbf{263} (2012), no.~11, 3564--3587.

\bibitem{hjorth_classification_2000}
Greg Hjorth, \emph{Classification and orbit equivalence relations},
  Mathematical Surveys and Monographs, vol.~75, American Mathematical Society,
  Providence, {RI}, 2000.

\bibitem{hjorth_complexity_2000}
Greg Hjorth and Alexander~S. Kechris, \emph{The complexity of the
  classification of {R}iemann surfaces and complex manifolds}, Illinois Journal
  of Mathematics \textbf{44} (2000), no.~1, 104--137.

\bibitem{kechris_countable_1992}
Alexander~S. Kechris, \emph{Countable sections for locally compact group
  actions}, Ergodic Theory and Dynamical Systems \textbf{12} (1992), no.~2,
  283--295.

\bibitem{kechris_classical_1995}
\bysame, \emph{Classical descriptive set theory}, Graduate Texts in
  Mathematics, vol. 156, Springer-Verlag, New York, 1995.

\bibitem{kerr_isomorphism_2013}
Matt Kerr, John~E. McCarthy, and Orr Shalit, \emph{On the isomorphism question
  for complete {P}ick multiplier algebras}, Integral Equations and Operator
  Theory \textbf{76} (2013), no.~1, 39--53.

\bibitem{lupini_polish_2014}
Martino Lupini, \emph{Polish groupoids and functorial complexity}, Transactions
  of the American Mathematical Society, to appear.

\bibitem{mccarthy_spaces_2015}
John~E. McCarthy and Orr~Moshe Shalit, \emph{Spaces of {D}irichlet series with
  the complete {P}ick property}, arXiv:1507.04162 (2015).

\bibitem{priestley_introduction_2003}
Hilary~A. Priestley, \emph{Introduction to complex analysis}, second ed.,
  Oxford University Press, Oxford, 2003.

\bibitem{ramsay_mackey-glimm_1990}
Arlan Ramsay, \emph{The {Mackey}-{Glimm} dichotomy for foliations and other
  {Polish} groupoids}, Journal of Functional Analysis \textbf{94} (1990),
  no.~2, 358--374.

\bibitem{ramsay_polish_2000}
Arlan~B. Ramsay, \emph{Polish groupoids}, Descriptive set theory and dynamical
  systems, London {Math}. {Soc}. {Lecture} {Note} {Ser}., vol. 277, Cambridge
  University Press, Cambridge, 2000, pp.~259--271.

\bibitem{resende_lectures_2006}
Pedro Resende, \emph{Lectures on \'{e}tale groupoids, inverse semigroups and
  quantales},  (2006).

\bibitem{rudin_function_2008}
Walter Rudin, \emph{Function theory in the unit ball of {$\mathbb{C}^n$}},
  Classics in {Mathematics}, Springer-Verlag, Berlin, 2008.

\bibitem{salomon_isomorphism_2014}
Guy Salomon and Orr Shalit, \emph{The isomorphism problem for complete {Pick}
  algebras: a survey}, arXiv:1412.7817 (2014).

\bibitem{shalit_operator_2013}
Orr Shalit, \emph{Operator theory and function theory in {D}rury-{A}rveson
  space and its quotients}, Operator Theory (Daniel Alpay, ed.), Springer,
  2015, pp.~1125--1180.

\bibitem{tsuji_potential_1975}
Matasugu Tsuji, \emph{Potential theory in modern function theory}, Chelsea
  Publishing Co., New York, 1975, Reprinting of the 1959 original.

\end{thebibliography}

\end{document}